\documentclass[12pt]{article}

\usepackage[utf8]{inputenc}
\usepackage{tabularx}
\usepackage{amsmath,amssymb,romannum}
\usepackage{longtable}
\usepackage{comment}
\usepackage{amsfonts,amsthm}
\usepackage[left=1in,top=1in,right=1in,bottom=1in,letterpaper]{geometry}
\usepackage{multirow,booktabs,makecell}
\usepackage{graphicx,epstopdf,subcaption}
\usepackage{standalone}
\usepackage{tikz}
\usetikzlibrary{positioning}
\usetikzlibrary{arrows.meta}
\usepackage{hyperref}
\usepackage{algorithm,algorithmic}
\usepackage{color}
\usepackage[sort&compress]{natbib}
\usepackage[normalem]{ulem}

\newtheorem{theorem}{Theorem}

\newtheorem{lemma}{Lemma}

\theoremstyle{remark}
\newtheorem{remark}{Remark}
\newtheorem{assumption}{Assumption}

\DeclareMathOperator*{\argmin}{arg\,min}

\title{Structure-Preserving Reduced-Order Modeling via Low-Rank Transport Signatures}

\author{
Jiajia Yu\thanks{yujiajia.math@gmail.com. Department of Mathematics, Duke University, Durham, NC 27710}
\and
Jingwei Hu\thanks{hujw@uw.edu. Department of Applied Mathematics, University of Washington, Seattle, WA 98195}
\and Fengyan Li\thanks{lif@rpi.edu. Department of Mathematical Sciences, Rensselaer Polytechnic Institute, Troy, NY 12180}
\and Shanyin Tong\thanks{tong3@sas.upenn.edu. Department of Mathematics, University of Pennsylvania, Philadelphia, PA 19104}
\and Yunan Yang\thanks{yunan.yang@cornell.edu. Department of Mathematics, Cornell University, Ithaca, NY 14850}
\and
Zhaiming Shen\thanks{zhaiming.shen@usu.edu.
Department of Mathematics and Statistics, Utah State University, Logan, UT 84322}
}

\date{}

\begin{document}

\maketitle

\pagenumbering{arabic}

\begin{abstract}
    Parametrized PDEs with density-valued solutions are often difficult to approximate with classical linear reduced-order models, especially in transport-dominated regimes. We introduce an optimal-transport-based reduced-order modeling that represents each density by the Kantorovich potential transporting a fixed reference density to the target density, and then maps these potentials to transport signatures using a weighted Laplacian associated with the reference measure. This embeds the density-valued solution map in a Hilbert space while preserving control of the induced transport maps and Wasserstein error. We treat the signature map as a continuous matrix indexed by parameters and space, construct a low-rank skeleton decomposition using a maximal-volume criterion, and learn the parameter-to-coefficient map with a neural network for efficient non-intrusive online evaluation. The reconstructed solution is obtained by pushing forward the reference density, so mass and positivity preservation is built into the method. We prove a mean-squared Wasserstein error bound separating low-rank approximation, discretization, sampling, and learning errors, and demonstrate the method on density-valued continuity equation and Burgers equation, where transport signatures yield substantially lower-rank structure than the original density snapshots. 
\end{abstract}

\section{Introduction}
Parametrized partial differential equations arise in many-query and real-time settings where a high-fidelity model must be evaluated repeatedly for many parameter values. Reduced-order modeling (ROM) seeks to replace such high-dimensional simulations with inexpensive surrogate models that retain the dominant dependence of the solution on the parameters. Classical approaches, including proper orthogonal decomposition (POD) and reduced-basis methods, construct a low-dimensional linear approximation space from representative snapshots and then approximate new solutions by projection or regression in this space~\citep{HesthavenRozzaStamm2016,BennerCohenOhlbergerWillcox2017,HesthavenPagliantiniRozza2022}. These methods have been highly successful for many elliptic, parabolic, and certain nonlinear problems, especially when the associated solution manifold has rapidly decaying Kolmogorov widths~\citep{DahmenPleskenWelper2014, Peherstorfer2020,FrescaDedeManzoni2021,HesthavenPeherstorferUnger2026}. 

However, transport-dominated problems pose a persistent challenge for linear reduced-order models. When coherent structures translate, deform, or concentrate, the solution manifold may be intrinsically low-dimensional but poorly approximated by a fixed linear subspace. A simple traveling wave, for instance, can require many linear modes despite being described by only a few physical parameters. This phenomenon is often referred to as the Kolmogorov-width barrier for transport-dominated problems. It has motivated a broad class of nonlinear ROM techniques, including shifted bases, registration methods, Lagrangian formulations, transported subspaces, adaptive bases, and neural-network-based nonlinear manifolds~\citep{RimPeherstorferMandli2023, NoninoBallarinRozzaMaday2023,Blickhan2024}.

The difficulty is even more pronounced when the state variable is a probability distribution or density. Linear combinations of density snapshots can destroy positivity and distort mass, and the associated linear geometry is insensitive to measure error caused by displacement.
For density-valued PDEs, such as continuity equations and Fokker--Planck equations, it is natural to compare solutions by the cost of transporting mass rather than by pointwise amplitude differences. Optimal transport provides precisely such a geometry through the Wasserstein distance, and it has become a central tool in the analysis and computation of probability measures~\citep{Villani2003, Santambrogio2015, PeyreCuturi2019,JordanKinderlehrerOtto1998}. 

Several reduced-order modeling strategies have recently incorporated ideas from optimal transport.  One line of work formulates model reduction directly in metric or Wasserstein spaces, using Wasserstein barycenters or tangent-space approximations to construct nonlinear reduced models for families of probability measures~\citep{EhrlacherLombardiMulaVialard2020,BattistiBlickhanEncheryEhrlacherLombardiMula2023}.  A related approach uses sparse Wasserstein barycenters to approximate and interpolate parametrized families of measures~\citep{DoFeydyMula2025}.  Optimal transport has also been combined with classical reduced-order modeling techniques, for example through POD-based approximations for kinetic equations~\citep{BernardIolloRiffaud2018}.  Another line of work uses displacement interpolation to generate physically meaningful interpolants between density snapshots or to augment limited training data~\citep{KhamlichPichiGirfoglioQuainiRozza2025}.  Optimal-transport-based registration methods have been used to align moving features before applying classical reduced-basis approximation~\citep{Blickhan2024}.  More broadly, linearized optimal transport and Monge-type embeddings provide Hilbert-space representations of probability measures through optimal maps from a reference distribution~\citep{WangSlepcevBasuOzolekRohde2013,KolouriTosunOzolekRohde2016,KolouriParkThorpeSlepcevRohde2017,MerigotDelalandeChazal2020}.  These works demonstrate that optimal transport can overcome some limitations of linear ROMs for transport-dominated data, while preserving important geometric and physical structure.

In this paper, we introduce an optimal-transport-based reduced-order modeling framework for probability distributions generated by parametrized partial differential equations. Let $\rho(\mu)$ denote the probability distribution associated with parameter value $\mu$, where time may be included among the parameters. Instead of constructing a reduced-order model directly for the densities $\rho(\mu)$, we first represent each distribution by the Kantorovich potential $\phi(\mu)$ transporting a fixed reference distribution $\bar\rho$ to $\rho(\mu)$. Thus the corresponding optimal transport map is
\[
    T_{\bar\rho\to\rho(\mu)}=I-\nabla \phi(\mu).
\]
We then define a transport signature
\[
     \mathbf{\Psi}(\mu,\cdot)=(I-\mathcal L)^{1/2}\phi(\mu), 
\]
where $\mathcal L=\nabla\cdot(\bar\rho\nabla)$ is a weighted Laplacian associated with the reference distribution. Working with the transport signature 
allows us to use the standard $L^2$ geometry for approximation while preserving a direct link to transport accuracy: the $L^2$ error in the signature controls the $L^2(\bar\rho)$ error of the corresponding optimal transport maps, and consequently the induced 2-Wasserstein error after pushforward. 

The central computational idea of our proposed reduced-order modeling is to regard $\mathbf{\Psi}(\mu,x)$ as a continuous matrix, or cmatrix, with rows indexed by the parameter $\mu$ and columns indexed by the spatial variable $x$~\citep{TownsendTrefethen2015}. We then construct a low-rank approximation in this transport-signature space using a skeleton decomposition based on a maximal volume criterion~\cite{goreinov1997theory,goreinov2001maximal,goreinov2010good}. The selected rows correspond to transport signatures at representative parameter values,
\[
    \psi_s(\cdot)=\mathbf{\Psi}(\mu_{i_s},\cdot), \qquad s=1,\ldots,r.
\]
For each training parameter $\mu_j$, $j=1,2,\cdots,n$, we compute coefficients by least-squares fitting on a spatial grid with $m$ points, obtaining an approximation of the form 
\[
    \mathbf{\Psi}(\mu_j,x)\approx \sum_{s=1}^r a_{m,s}(\mu_j)\psi_s(x).
\]
A neural network is then trained to learn the parameter-to-coefficient map 
\[
    \mu \mapsto a_m(\mu).
\]
For a new parameter value $\mu$, the online stage evaluates the learned coefficients, reconstructs the transport signature, maps it back to an approximate potential, and differentiates to obtain an approximate transport map $\widehat {\mathbf{T}}_{n,m}^r$,  
and finally pushes forward the reference distribution:
\[
    \widehat\rho_{n,m}^r(\mu)
    =
    \bigl(\widehat {\mathbf{T}}_{n,m}^r(\mu,\cdot)\bigr)_{\#}\bar\rho .
\]
Thus the final output is again a probability distribution, with the preservation of mass and the positivity built into the reconstruction.

The proposed method differs from existing optimal-transport ROMs in several respects.  Wasserstein barycenter approaches typically require the solution of a barycenter problem over a dictionary of measures during reconstruction, whereas our online stage only evaluates a learned coefficient map and pushes forward the reference distribution.  Displacement interpolation methods construct approximations along Wasserstein geodesics between selected snapshots. In contrast, our method builds a global reduced representation in transport-signature space. Similarly, instead of using optimal transport only as a registration or alignment step before applying a classical reduced-basis method, we construct the reduced-order model directly from transport signatures associated with the density-valued solution map.  The resulting skeleton decomposition selects representative signatures and yields an interpretable low-rank structure, while the learned parameter-to-coefficient map provides a non-intrusive online evaluation procedure. 

Our main theoretical result gives a mean-squared $W_2$ error bound for the reconstructed distribution $\widehat\rho_{n,m}^r(\mu)$ relative to the true distribution $\rho(\mu)$. The bound separates the principal sources of error: the low-rank approximation error in transport-signature space, the effect of spatial discretization and quadrature, the coverage of the parameter samples, and the statistical error from learning the coefficient map. This analysis connects approximation in the Hilbert space of signatures to error in the natural Wasserstein geometry of probability distributions.

We demonstrate the framework on the density-valued continuity equation in two spatial dimensions and the inviscid and viscous Burgers' equation in one spatial dimension. The results show that transport signatures can exhibit substantially more favorable low-rank structure than the original density snapshots, leading to accurate reduced-order approximations with a small number of selected transport-based basis. In addition, our method naturally preserves the positivity and total mass of the density.

The rest of the paper is organized as follows. Section~\ref{sec:background} reviews the background material needed for the proposed ROM method, including optimal transport theory, computational aspects of optimal transport, and the skeleton decomposition algorithm. Section~\ref{sec:otrom} presents the transport-signature-based ROM method, and Section~\ref{sec:theory} provides a rigorous theoretical error analysis. Numerical examples are given in Section~\ref{sec:num}, followed by concluding remarks in Section~\ref{sec:conclusions}.

\section{Background}~\label{sec:background}
This section provides essential background for the proposed optimal transport-based reduced-order method.

\subsection{Optimal transport}
Let $\Omega$ be a subset of the Euclidean space and $\mathbb{P}(\Omega)$ be the set of probability distributions on $\Omega$. 
For given $\rho_0,\rho_1\in\mathbb{P}(\Omega)$, the goal of optimal transport is to find the map that moves $\rho_0$ to $\rho_1$ with the least effort. 
Given a cost function $c:\Omega\times\Omega\to\mathbb{R}$, the Monge problem seeks a measurable map $T:\Omega\to\Omega$ solving
\begin{equation}
    \inf_{T} \left\{ \int_{\Omega} c(x,T(x)) \, d \rho_0(x) : T_\#\rho_0=\rho_1 \right\},
    \label{eq:monge}
\end{equation}
where $T_\#\rho_0$ denotes the push-forward measure of $\rho_0$ by $T$. 
More precisely, 
\begin{equation}
    (T_\#\rho_0)(B) = \rho_0(T^{-1}(B)) \quad\text{for any measurable $B\subset\Omega$}
\end{equation}
A minimizer of \eqref{eq:monge}, when it exists, is called an \emph{optimal transport map}.

\paragraph{Potentials and transport signature}

Since the Monge problem imposes that mass moves deterministically through a map, an optimal transport map may fail to exist in general. A relaxation is the Kantorovich problem optimizing over transport plans:
\begin{equation}
    \inf_{\pi\in\Pi(\rho_0,\rho_1)}
    \int_{\Omega\times\Omega} c(x,y)\,d\pi(x,y).
\label{eq:kantorovich}
\end{equation}
Here, $\Pi(\rho_0,\rho_1)$ denotes the set of probability measures on $\Omega\times\Omega$ whose marginals are $\rho_0$ and $\rho_1$. 
The Kantorovich formulation gives rise to the \emph{$p$-Wasserstein distance}
\begin{equation}
    W_p(\rho_0,\rho_1):=
    \left( \inf_{\pi\in\Pi(\rho_0,\rho_1)}
    \int_{\Omega\times\Omega} \|x-y\|_p^p\,d\pi(x,y) \right)^{1/p}.
\end{equation}

For the quadratic cost $c(x,y)=\frac{1}{2}\|x-y\|_2^2$, the Kantorovich problem~\eqref{eq:kantorovich} has the following equivalent dual formulation:
\begin{equation}
\sup_{\phi,\eta}
\left\{ \int_{\Omega} \phi(x)\,d\rho_0(x)
+ \int_{\Omega} \eta(y)\,d\rho_1(y):
\phi(x)+\eta(y)\leq \frac{1}{2}\|x-y\|_2^2 \right\}.
\label{eq:kantorovich dual}
\end{equation}
The maximizers $\phi$ and $\eta$, when they exist, are called \emph{Kantorovich potentials}. 

Let $\mathbb{P}_2(\mathbb{R}^d)$ be the set of probability distributions with finite second moment.
A fundamental result in optimal transport (\cite[Thm.~1.22]{Santambrogio2015} \cite{brenier1991polar}) shows that, for the quadratic cost, the optimal transport map exists and is induced by a convex potential function.

\begin{theorem}[Brenier Theorem]\label{thm:Brenier}
Let $\rho_0,\rho_1\in \mathbb{P}_2(\mathbb{R}^d)$ and assume that $\rho_0$ is absolutely continuous with respect to Lebesgue measure.
For the quadratic cost $c(x,y)=\frac{1}{2}\|x-y\|_2^2$, there exists a $\rho_0$-a.e. unique optimal transport map $T$ pushing $\rho_0$ forward to $\rho_1$. 
Moreover, there exists a convex function $u:\mathbb{R}^d\to\mathbb{R}$, unique up to a change that does not affect its gradient on the support of $\rho_0$, 
such that
$$ T=\nabla u, \qquad \rho_0\text{-a.e.}. $$
\end{theorem}
The convex function $u$ is often called the \emph{optimal transport potential}, also known as the \emph{Brenier potential}.
In particular, $u$ is associated with the Kantorovich potential $\phi$ under the assumptions of Brenier's theorem (Theorem~\ref{thm:Brenier}) via 
\begin{equation}
    u(x)=\frac{1}{2}\|x\|_2^2-\phi(x).
\label{eq:kan-bre}
\end{equation}

Let $(\mathbb{P}_2(\mathbb{R}^d),W_2)$ denote the Wasserstein-2 space.
At an absolutely continuous measure $\bar{\rho}$, one can define the tangent cone to $\bar{\rho}$ in $(\mathbb{P}_2(\mathbb{R}^d),W_2)$ using optimal transport maps following~\cite[Eq.~8.5.1]{AGS2008gradflow}. 
\begin{equation}
    \operatorname{Tan}_{\bar{\rho}} \mathbb{P}_2(\mathbb{R}^d)
    := \overline{ \{ \lambda(T-I): \text{$T$ is the optimal transport map from $\bar{\rho}$ to $T_\#\bar{\rho}$}, \; \lambda>0 \} }^{L^2_{\bar{\rho}}(\mathbb{R}^d;\mathbb{R}^d)}
\end{equation}
According to~\cite[Thm.~8.5.1]{AGS2008gradflow}, the tangent cone is in fact a linear subspace of $L^2_{\bar{\rho}}(\mathbb{R}^d;\mathbb{R}^d)$:
\begin{equation}
    \operatorname{Tan}_{\bar{\rho}} \mathbb{P}_2(\mathbb{R}^d)
    = \overline{\{\nabla \phi : \phi \in C_c^\infty(\mathbb R^d)\}}^{L^2_{\bar{\rho}}(\mathbb{R}^d;\mathbb{R}^d)}. 
\end{equation}
The tangent vectors can be interpreted as the optimal transport velocity fields, namely the velocity fields with minimal kinetic energy among all fields generating the same infinitesimal variation of the density.
As illustrated in Figure~\ref{fig:illu wtangent}, the velocity field $-\nabla\phi$ transports $\bar{\rho}$ to $\rho$ through the continuity equation and minimizes the kinetic energy among all velocity fields achieving this transport.
\begin{figure}[htb]
    \centering
    \begin{tikzpicture}[scale=1.0,
    manifold/.style={draw=black, very thick, fill=gray!20},
    tangentplane/.style={draw=black, very thick, fill=white, fill opacity=0.15},
    >=Latex,
    every node/.style={font=\small}
]

% =========================
% Left: Riemannian manifold
% =========================

% Manifold surface
\path[manifold]
(-6.8,0.2)
.. controls (-4.5,3.0) and (-2.5,3.3) ..
(-0.3,0.1)
.. controls (-2.5,0.6) ..
(-3.5,-0.3)
.. controls (-5.2,0.8) ..
cycle;

\node[above] at (-3.5,0.2) {$\mathbb{P}_2(\mathbb{R}^d)$};

% Tangent plane
\coordinate (P1) at ( 0.60,3.35);
\coordinate (P2) at (-1.75,1.25);
\coordinate (P3) at (-7.00,1.25);
\coordinate (P4) at (-4.35,3.35);

\path[tangentplane]
    (P1) -- (P2) -- (P3) -- (P4) -- cycle;

\node[below] at (-0.95, 3.35) {$\operatorname{Tan}_{\bar{\rho}} \mathbb{P}_2(\mathbb{R}^d)$};

% Point on manifold
\coordinate (rhobar) at (-3.70,2.05);
\fill[opacity=0.5] (rhobar) circle (2.5pt);
\node[below] at (-3.70,1.97) {$\bar{\rho}$};
\fill[opacity=0.5] (-2.10,1.45) rectangle ++(5pt,5pt);
\node[left] at (-2.10,1.45) {$\rho$};

% Curve on manifold
\draw[very thick, dashed] (-3.70,2.05)
.. controls (-3.00,2.10) ..
    (-2.10,1.55);

% Tangent vector
\draw[->,very thick]
(rhobar) -- (-2.20,2.45);
\node[above] at (-2.70,2.35) {$-\nabla\phi$};

% \node[above right] at (3.35,2.15) {$\xi\in T_xM$};

% =========================
% Right: Wasserstein tangent
% =========================

% Fluid domain
\draw[thick,rounded corners,fill=gray!70]
(2.0,0.1) rectangle (7.5,3.2);

% Fluid particles of rhobar

\foreach \x/\y in {
2.6/0.7,
3.2/1.7,
3.7/1.0,
4.7/1.9,
5.6/0.6,
6.1/2.3
}{
    \fill[opacity=0.5] (\x,\y) circle (2.5pt);
}

% Velocity field arrows
\draw[->,thick]
(2.6,0.7) -- (3.2,1.1);
\draw[->,thick]
(3.2,1.7) -- (3.7,2.0);
\draw[->,thick]
(3.7,1.0) -- (4.2,1.5);
\draw[->,thick]
(4.7,1.9) -- (5.2,2.4);
\draw[->,thick]
(5.6,0.6) -- (6.1,1.4);
\draw[->,thick]
(6.1,2.3) -- (6.7,2.7);

% Fluid particles of rhoepsilon
\foreach \x/\y in {
3.1/1.0,
3.6/1.9,
4.1/1.4,
5.1/2.3,
6.0/1.3,
6.6/2.6
}{
    \fill[opacity=0.5] (\x,\y) rectangle ++(5pt,5pt);
}

% Wasserstein notation
\node at (4.55,-0.2)
{$x\mapsto x-\nabla\phi(x)$};
% {$\dot{\rho}=-\nabla\cdot(\rho\nabla\varphi),\,\rho_0=\bar{\rho}$};

\end{tikzpicture}
\caption{Illustration of the Wasserstein tangent space and tangent vector.}
\label{fig:illu wtangent}
\end{figure}

This characterization motivates the construction of a reduced-order model in the tangent space. 
Since the optimal transport displacement satisfies $T-I=-\nabla\phi$, where $\phi$ is a Kantorovich potential, the transport information is encoded in $\phi$. We therefore use Kantorovich potentials to derive transport signatures and construct the reduced-order model.

\paragraph{Computation: exact formula and the back-and-forth method}
The  Kantorovich potential and the optimal transport map will be computed by exact formulas in one dimension and by the iterative back-and-forth method in high dimensions.

For a probability density function $\rho$ on $\mathbb{R}$, let the cumulative distribution function be
\begin{equation}
    F_\rho:\mathbb{R}\to[0,1],\quad
    x\mapsto\int_{-\infty}^x \rho(y)\,dy.
\end{equation}
The pseudoinverse of $F_\rho$ is called the quantile function and is defined by
\begin{equation}
    F^{-1}_\rho:[0,1]\to\mathbb{R}\cup\{-\infty\},
    \quad
    r\mapsto \inf\{x\in\mathbb{R}:F_\rho(x)\geq r\}.
\end{equation}
For two probability densities $\rho_0$ and $\rho_1$ on $\mathbb{R}$, the optimal transport map from $\rho_0$ to $\rho_1$ is given by the monotone rearrangement
\begin{equation}
    T=F_{\rho_1}^{-1}\circ F_{\rho_0}.
\end{equation}
Consequently, the squared Wasserstein-2 distance can be computed as
\begin{equation}
    W_2^2(\rho_0,\rho_1)
    =\int_0^1
    \left|F_{\rho_1}^{-1}(r)-F_{\rho_0}^{-1}(r)\right|^2\,dr.
\end{equation}
The Kantorovich potential can also be explicitly recovered from the optimal transport map by integration, up to an additive constant,
\begin{equation}
    \varphi(x)
    =\int_{x_0}^x
    \left(
        s-F_{\rho_1}^{-1}(F_{\rho_0}(s))
    \right)\,ds.
\label{eq:kan 1d exact}
\end{equation}

When the PDE solution has dimension higher than one, we compute the optimal transport map using the back-and-forth method of \citep{bfm}. It is based on the Kantorovich dual formulation~\eqref{eq:kantorovich dual}.

Define the $c$-transform of $\phi$ and $\eta$ as
\begin{equation}
\phi^c(y):=\inf_{x\in\Omega}\left\{c(x,y)-\phi(x)\right\},
\qquad \eta^c(x):=\inf_{y\in\Omega}\left\{c(x,y)-\eta(y)\right\},
\end{equation}
Instead of optimizing over all pairs $(\phi,\eta)$ satisfying the inequality constraint, one can optimize over a single potential at a time: if $\eta$ is fixed, then the best admissible choice of $\phi$ is $\phi=\eta^c$, and vice versa.
More concretely, define the reduced dual objectives
\begin{equation}
\mathcal{I}(\phi):=
\int_{\Omega} \phi(x)\,d\rho_0(x)
+\int_{\Omega} \phi^c(y)\,d\rho_1(y),
\qquad \mathcal{J}(\eta):=
\int_{\Omega} \eta^c(x)\,d\rho_0(x)
+\int_{\Omega} \eta(y)\,d\rho_1(y).
\end{equation}
The method moves ``back and forth’’ between the two dual variables, using the $c$-transform to maintain admissibility while using gradient ascent to increase the dual objective.
Given a step size $\sigma>0$, one iteration has the form
\begin{align}
    &\phi_{n+\frac{1}{2}}
    =\phi_{n}+\sigma \nabla_{\dot H^1} \mathcal{I}(\phi_{n}), 
    \qquad\qquad  \eta_{n+\frac{1}{2}}
    =\phi_{n+\frac{1}{2}}^c,\\
    &\eta_{n+1}
    =\eta_{n+\frac{1}{2}}+\sigma \nabla_{\dot H^1} \mathcal{J}(\eta_{n+\frac{1}{2}}), 
    \qquad \phi_{n+1}
    =\eta_{n+1}^c.
\end{align}
Denote $T_\phi=I-\nabla \phi^c$ and $T_\eta=I-\nabla \eta^c$, the homogeneous $\dot H^1$ gradients (as defined in \citep{bfm}) are computed by
\begin{equation}
    \nabla_{\dot H^1}\mathcal{I}(\phi)
    =(-\Delta)^{-1} \left(\rho_0-(T_\phi)_\#\rho_1\right),
    \qquad \nabla_{\dot H^1}\mathcal{J}(\eta)
    =(-\Delta)^{-1} \left(\rho_1-(T_\eta)_\#\rho_0\right).
\end{equation}
The update increases the potential in regions where the current map transports too little mass and decreases it where too much mass is transported. The inverse Laplacian acts as a preconditioner, producing a smooth correction to the potential.

With the quadratic cost $c(x,y)=\frac{1}{2}|x-y|^2$, once the dual variables $\phi$ and $\eta$ are obtained, the squared Wasserstein-2 distance can be computed as 
\begin{equation}
    W_2^2(\rho_0,\rho_1) = 2 \left( \int_{\Omega}\phi(x)\,d\rho_0(x) + \int_{\Omega}\eta(y)\,d\rho_1(y) \right).
\end{equation}

From a computational point of view, the method is efficient on regular grids. Each iteration requires two main operations: computing $c$-transforms and solving Poisson-type equations arising from the $\dot H^1$ gradient. For quadratic costs, the $c$-transforms in high dimensions can be decomposed into the Legendre transformation in one dimension, and therefore can be computed in nearly linear time~\cite{bfm,lucet1997legendre}.  
The inverse Laplacian can be applied efficiently using fast Fourier transforms or related elliptic solvers. As a result, for $m$ grid points, the method has approximately $O(m\log m)$ computational cost per iteration and $O(m)$ memory usage. This makes the back-and-forth method suitable for large-scale problems where one needs not only the Wasserstein distance, but also an accurate approximation of the optimal transport map $T$. 

\subsection{Skeleton decomposition}

Next, let us introduce the skeleton decomposition. Given a matrix $A\in \mathbb{R}^{n\times m}$, letting  $I = \{i_1,\cdots,   i_r\}$ and $J = \{j_1, \cdots,  j_r\}$  be the row and column indices,   we denote by $A_{I,:}$  all the rows of  $A$ with  row indices in $I$ and $A_{:,J}$  the all the columns of $A$ with column indices in $J$. Assume that the submatrix $A_{I;J}$ situated on rows $I$ and columns $J$ is nonsingular. Then 
\begin{equation}
\label{skeleton}
A_{:,J} A_{I,J}^{-1} A_{I,:} \approx A
\end{equation}
is called a matrix skeleton or cross approximation~\cite{goreinov1997theory,goreinov2001maximal,goreinov2010good}. Note that the left-hand side of~\eqref{skeleton} is also called CUR decomposition in the literature. One notable advantage of skeleton decomposition over classical SVD is interpretability. While SVD expresses data in an abstract basis (e.g., eigenspaces), obscuring the physical meaning of the original rows and columns, skeleton decomposition retains actual rows and columns of the matrix, preserving their physical interpretation. Additionally, SVD becomes computationally prohibitive for large-scale problems, whereas skeleton decomposition offers a more scalable alternative.

More recently, it is shown in \citep{ALS} that when $A_{I,J}$ has the maximal volume (absolute value of the determinant) among all $r \times r$ submatrices of $A$, then the Chebyshev (maximum) norm of the residual matrix satisfies
    \begin{equation}
\label{myskapp}
\left\|A - A_{:,J}A_{I,J}^{-1}A_{I,:}\right\|_C \le \frac{r+1}{\sqrt{\sum_{k=1}^{r+1}\sigma_k^{-2}(A)}},
\end{equation}
where $\sigma_{k}(A)$ is the $k$-th singular value of $A$.

Algorithmically, a greedy-based iterative approach was proposed in \citep{ALS} to effectively select the rows and columns based on the maximal volume criterion. We summarize it in Algorithm~\ref{2dmvalg} and this will be used in our analysis.   For the actual implementation of our
proposed ROM, it suffices to select indices along only one direction (either rows or columns, as will be explained in Section~\ref{sec:otrom}). We therefore specialize Algorithm~\ref{2dmvalg} to the symmetric matrix setting, where the same indices are selected for both rows and columns. The new procedure for selecting symmetric indices based on the maximal volume criterion is summarized in Algorithm~\ref{2dmvalg_symm}.

\begin{algorithm}[H]
\caption{Skeleton Decomposition via Maximal Volume \label{2dmvalg}}
\begin{algorithmic}[1]
\REQUIRE matrix to decompose: $n \times m$ matrix $M$; 
initialization: row index set $I_0$, column index set $J_0$ and $r\times r$ submatrix $A_0:=M(I_0,J_0)$ with $\det(A_0) \neq 0$; 
tolerance $\epsilon > 0$;
$b_{ij} = \infty,(i,j=1,2,\cdots,n)$.
\ENSURE output row index set $I_l$, column index set $J_l$ and $A_l:=M(I_l,J_l)$ such that $A_l$ is a close to dominant submatrix.
\STATE{Set $l=0$.}
\WHILE{$\max_{i,j=1,2,\cdots,n}|b_{ij}| > 1 + \epsilon$}
\STATE{Let $B_l = M(:,J_l)A_l^{-1}$, and $C_l = A_l^{-1}M(I_l,:)$.}
\STATE{Set $b_{ij}$ equal to the largest in modulus entry of both $B_l$ and $C_l$.}
\IF{$b_{ij}$ is from $B_l$}
\STATE{Replace the $j$th row of $A_l$ with the $i$th row of $M(:,J_l)$.}
\ENDIF
\IF{$b_{ij}$ is from $C_l$}
\STATE{Replace the $i$th column of $A_l$ with the $j$th column of $M(I_l,:)$.}
\ENDIF
\STATE{Update index: $I_{l+1}\leftarrow I_l,A_{l+1}\leftarrow A_l, l\leftarrow l+1$.}
\ENDWHILE
\end{algorithmic}
\end{algorithm}

\begin{algorithm}[H]
\caption{Symmetric Skeleton Decomposition via Maximal Volume \label{2dmvalg_symm}}
\begin{algorithmic}[1]
\REQUIRE matrix to decompose: $n \times n$ symmetric matrix $M$; 
initialization: row and column index set $I_0$ and $r\times r$ submatrix $A_0:=M(I_0,I_0)$ with $\det(A_0) \neq 0$;
tolerance $\epsilon > 0$;
$b_{ij} = \infty$, $(i,j=1,2,\cdots,n)$.
\ENSURE output row and column index set $I_l$ and $A_l:=M(I_l,I_l)$ such that $A_l$ is a close to dominant submatrix.
\STATE{Set $l=0$.}
\WHILE{$\max_{i,j=1,\cdots,n}|b_{ij}| > 1 + \epsilon$}
\STATE{Let $B_l = M(:,I_l)A_l^{-1}$, and $C_l = A_l^{-1}M(I_l,:)$.}
\STATE{Set $b_{ij}$ equal to the largest in modulus entry of both $B_l$ and $C_l$.}
\IF{$b_{ij}$ is from $B_l$}
\STATE{Replace the $j$th row of $A_l$ with the $i$th row of $M(:,I_l)$.}
\STATE{Replace the $j$th column of $A_l$ with the $i$th column of $M(I_l,:)$. %
}
\ENDIF
\IF{$b_{ij}$ is from $C_l$}
\STATE{Replace the $i$th column of $A_l$ with the $j$th column of $M(I_l,:)$.}
\STATE{Replace the $i$th row of $A_l$ with the $j$th row of $M(:,I_l)$. %
}
\ENDIF
\STATE{ Update index: $ I_{l+1}\leftarrow I_{l}, A_{l+1}\leftarrow A_l, l\leftarrow l+1$.}
\ENDWHILE
\end{algorithmic}
\end{algorithm}

\section{ROM based on optimal transport}
\label{sec:otrom}

Let $\mu$ denote a scalar or vector-valued parameter,  including time when applicable, and suppose that $\mu$ has law $Q$. We assume that $Q$ is supported on a compact, connected, $d$-dimensional Riemannian submanifold $\mathcal{M}\subseteq [-b,b]^D$, with $d\leq D$, and that $Q$ admits a density with respect to the Riemannian volume measure on $\mathcal{M}$ that is bounded above and below by positive constants. For each $\mu\in\mathcal{M}$, let $\rho(\mu)$ denote the probability measure induced by the corresponding parameterized partial differential equation. Let
\[
\Omega := [-\xi,\xi]^{D_{\bar{\rho}}},
\]
and assume that $\operatorname{supp}(\rho(\mu))\subseteq \Omega$ for every $\mu\in\mathcal{M}$. Throughout this section, $dx$ denotes the normalized Lebesgue measure on $\Omega$.  Finally, we assume that the reference distribution $\bar{\rho}\in\mathcal{P}(\Omega)$ is absolutely continuous with respect to $dx$; by a standard abuse of notation, its density is again denoted by $\bar{\rho}$.
Throughout, for a measure $\mathfrak m$ on $X$, we denote by $L^2_{\mathfrak m}(X;Y)$ the space of measurable functions $f:X\to Y$ such that
$\int_X |f(x)|^2\,d\mathfrak m(x)<\infty$.
By the same abuse of notation as above, we identify an absolutely continuous measure with its density. When $\mathfrak m$ is the Lebesgue measure, we omit the subscript $\mathfrak m$, and when $Y=\mathbb{R}$, we omit the codomain $Y$.

\begin{figure}[htb]
    \centering
    \includegraphics[width=1.0\textwidth]{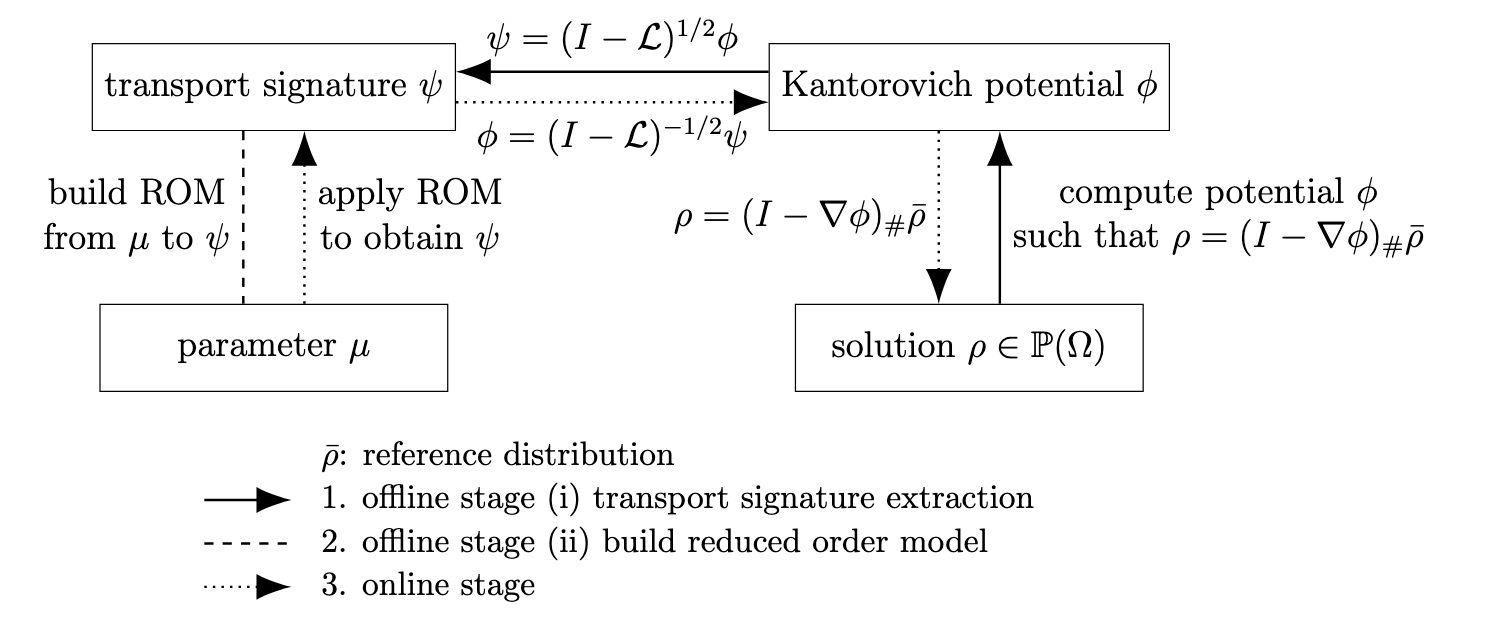}
\caption{Illustration of the OT-based reduced-order model.%
}
\label{fig:illu rom}
\end{figure}

Figure~\ref{fig:illu rom} and Algorithm~\ref{algROMskeleton} summarize the proposed reduced-order modeling framework. Instead of constructing a reduced-order model directly for the solution distributions $\rho(\mu)$, we first map these distributions into a space of transport signatures. For each parameter value $\mu$, we compute the optimal transport potential $\phi(\mu)$ from the reference distribution $\bar{\rho}$ to $\rho(\mu)$, and use this potential to define the transport signature $\mathbf{\Psi}(\mu,\cdot)$, as detailed later in this section. This representation is motivated by the observation that, for transport-dominated problems, transport signatures often possess a substantially more favorable low-rank structure than the original distributions, leading to more efficient reduced-order approximations.

Our OT-based ROM procedure consists of three main stages:
\begin{enumerate}
    \item \emph{Offline stage (\romannum{1}) } (Steps 1 - 3 in Algorithm~\ref{algROMskeleton}): For any $\mu\in \mathcal{M}$, compute the Kantorovich potential $\phi(\mu)$ and transport signature $\psi(\mu)$ and build the map $\mathbf{\Phi}(\mu,\cdot):=\phi(\mu)$ and $\mathbf{\Psi}(\mu,\cdot):=\psi(\mu)$.
    In practice, the computation is done only for a sampled training parameter set $\{\mu_j\}_{j=1}^n$.
    \item \emph{Offline stage (\romannum{2})} (Steps 4 - 7 in Algorithm~\ref{algROMskeleton}): constructing a reduced-order model in the resulting transport-signature space, i.e., for $\mathbf{\Psi}(\mu,\cdot)$; and
    \item \emph{Online stage} (Step 8 in Algorithm~\ref{algROMskeleton}): predicting the transport signature $\mathbf{\Psi}(\mu,\cdot)$ associated with a new parameter value $\mu$, and reconstructing the target probability distribution $\rho(\mu)$ from the predicted signature.
\end{enumerate}
The remainder of this section describes these steps in detail.

Table~\ref{tab:ot-rom-notation} summarizes key notations in our algorithm and analysis.

\subsection{Offline stage (\romannum{1}): signature extraction}
\label{subsec: off1}

For any $\mu\in\mathcal{M}$, let $\phi(\mu)$ be the Kantorovich potential (see Theorem~\ref{thm:Brenier}) for the optimal transport from $\bar{\rho}$ to $\rho(\mu)$, normalized by
\[
    \int_\Omega \phi(\mu)(x)\,dx=0.
\]
Thus, for the quadratic transport cost, the optimal transport map pushing $\bar{\rho}$ forward to $\rho(\mu)$ is
\[
    T_{\bar{\rho}\to\rho(\mu)}=I-\nabla \phi(\mu).
\]
We collect these Kantorovich potentials and maps into the parameter-dependent functions
\begin{equation}
    \mathbf{\Phi}:\mathcal{M}\times\Omega\to\mathbb{R},
    \qquad 
    \mathbf{\Phi}(\mu,\cdot):=\phi(\mu),
\end{equation}
and
\begin{equation}
    \mathbf{T}:\mathcal{M}\times\Omega\to\mathbb{R}^{D_{\bar{\rho}}},
    \qquad 
    \mathbf{T}(\mu,\cdot):=T_{\bar{\rho}\to\rho(\mu)}.
\end{equation}
By construction,
\[
    \mathbf{Id}_x-\nabla_x \mathbf{\Phi}=\mathbf{T}.
\]

We define the weighted Laplacian acting on functions of $x$ by
\begin{equation}\label{eq:L}
    \mathcal{L}f := \nabla_x\cdot(\bar{\rho}\,\nabla_x f),
\end{equation}
equipped with a self-adjoint realization on $L^2(\Omega)$ corresponding to zero-Neumann boundary condition 
With this sign convention, $\mathcal{L}$ is nonpositive, and hence $I-\mathcal{L}$ is positive and self-adjoint on $L^2(\Omega)$. Therefore, $(I-\mathcal{L})^{1/2}$ is well-defined by the spectral theorem. Moreover,
\begin{equation}\label{eq:H1seminormbd}
    \|(I-\mathcal L)^{1/2}f\|_{L^2(\Omega)}^2
    =
    \|f\|_{L^2(\Omega)}^2
    +
    \|\nabla_x f\|_{L^2_{\bar\rho}(\Omega)}^2
    \geq
    \|\nabla_x f\|_{L^2_{\bar\rho}(\Omega)}^2.
\end{equation}
Equivalently,
\[
    \|(I-\mathcal L)^{1/2}f\|_{L^2(\Omega)}
    \geq
    \|\nabla_x f\|_{L^2_{\bar\rho}(\Omega)}.
\]

We define the \emph{transport signature} by
\begin{equation}\label{eq:Psi}
    \mathbf{\Psi}:\mathcal{M}\times\Omega\to\mathbb{R},
    \qquad
    \mathbf{\Psi}:=(I-\mathcal{L})^{1/2}\mathbf{\Phi}.
\end{equation}
The reduced-order model is then constructed in the transport-signature space, using $\mathbf{\Psi}$ rather than $\mathbf{\Phi}$. This choice is useful because approximation errors in $\mathbf{\Psi}$ directly control the corresponding $L^2_{\bar\rho}(\Omega)$ errors in the transport maps $\mathbf{T}=\mathbf{Id}_x-\nabla_x\mathbf{\Phi}$.

\begin{algorithm}[t]
\caption{ROM via Skeleton Decomposition} \label{algROMskeleton}
\begin{algorithmic}[1]
\REQUIRE PDE solution distributions $\rho(\mu_j)$ at parameter values $\{\mu_j\}_{j=1}^{n}$, reference distribution $\bar{\rho}$, spatial grid $\{x_k\}_{k=1}^m$, unseen parameter $\mu$, and target rank $r$.
\ENSURE Approximation $\widehat{\rho}_{n,m}^r(\mu)$ to the solution distribution $\rho(\mu)$.

\STATE{Compute the Kantorovich potential $\mathbf{\Phi}(\mu_j,\cdot)$ for which the optimal transport map satisfies 
$T_{\bar \rho \rightarrow \rho(\mu_j)} = I-\nabla_x \mathbf{\Phi}(\mu_j,\cdot)$, for $j=1,\ldots,n$.}

\STATE{Let $\mathbf{\Psi}(\mu_j,\cdot):=(I-\mathcal{L})^{1/2}\mathbf{\Phi}(\mu_j,\cdot)$, for $j=1,\ldots,n$.}

\STATE{Form the sampled signature matrix $\Psi\in\mathbb{R}^{n\times m}$ with entries
$\Psi_{j,k}:=\mathbf{\Psi}(\mu_j,x_k)$.}

\STATE{$\Psi\Psi^{\top}\xrightarrow{\text{symmetric skeleton decomposition (Algorithm \ref{2dmvalg_symm})}}\text{row indices }\{i_s\}_{s=1}^{r}$, and let $\psi_{s}(\cdot):=\mathbf{\Psi}(\mu_{i_s},\cdot)$}

\STATE{For $j=1,\cdots,n$, find the grid-based coefficients $a_m(\mu_j)\in\mathbb{R}^r$ such that 
\[
\mathbf{\Psi}(\mu_j,\cdot)
\approx
\sum_{s=1}^{r} a_{m,s}(\mu_j)\psi_{s}(\cdot)
\]
by least-squares fitting on the spatial grid $\{x_k\}_{k=1}^m$.} 

\STATE{Learn the parameter-to-coefficient map 
$\mu\mapsto a_{m}(\mu)$ empirically via a feed-forward neural network to obtain the empirical risk minimizer 
$\widehat{a}_m$.}
\STATE{Form the rank-$r$ model
\[
\widehat{\mathbf{\Psi}}_{n,m}^{r}(\mu,x)
:=
\sum_{s=1}^{r}\widehat{a}_{m,s}(\mu)\,\psi_{s}(x),
\]
and define $
\widehat{\mathbf{\Phi}}_{n,m}^r
:=
(I-\mathcal L)^{-1/2}\widehat{\mathbf{\Psi}}_{n,m}^r$ and $
\widehat{\mathbf{T}}_{n,m}^r(\mu,x)
= x-\nabla_x \widehat{\mathbf{\Phi}}_{n,m}^r(\mu,x)$.
}
\STATE{Obtain the approximation to the density at parameter $\mu$ by $\widehat{\rho}_{n,m}^r(\mu)
:=(\widehat{\mathbf{T}}_{n,m}^r(\mu,\cdot))_{\#}\bar\rho$.}

\end{algorithmic}
\end{algorithm}

\begin{remark}[Sample splitting in the analysis] Algorithm~\ref{algROMskeleton} is written using a single offline parameter set for simplicity. In the theoretical analysis below, we use a sample-splitting convention: one parameter set is used to select the skeleton basis, and another independent parameter set is used to train the coefficient map. This separates the deterministic basis-selection error from the statistical regression error. In computations, the same data set may be used 
for both tasks, but the theorem is stated for the split-sample version. 
\end{remark}

\subsection{Offline stage (\romannum{2}): reduced-order model construction}
\label{subsec: off2}

Throughout this paper, we regard $\mathbf{\Psi}$ as a continuous matrix, or cmatrix, whose rows are indexed by $\mu\in\mathcal{M}$ and whose columns are indexed by $x\in\Omega$. Equivalently, $\mathbf{\Psi}$ is the kernel of the linear operator
\begin{equation}\label{eq:H}
H_\Psi : L^2(\Omega) \rightarrow L^2_Q(\mathcal{M}),
\qquad
(H_\Psi f)(\mu)=\int_\Omega \mathbf{\Psi}(\mu,x) f(x)\, dx.
\end{equation}
For a fixed parameter value $\mu\in\mathcal M$, the function $\mathbf{\Psi}(\mu,\cdot)\in L^2(\Omega)$ is viewed as a row of the cmatrix. Similarly, for a fixed spatial point $x\in\Omega$, the function $\mathbf{\Psi}(\cdot,x)\in L^2_Q(\mathcal{M})$ is viewed as a column. For index sets $I \subset\mathcal{M}$ and $J \subset\Omega$, the restrictions $\mathbf{\Psi}(I,\cdot)$ and $\mathbf{\Psi}(\cdot,J)$ play the role of row and column quasi-matrices in the sense of \citep{TownsendTrefethen2015}. Hereafter, we will abuse the notation to use $I$ and $J$ to denote both the index sets and the corresponding point subsets in $\mathcal{M}$ and $\Omega$, respectively.

\paragraph{Reduced-order basis}
Let $\nu_1,\ldots,\nu_r\in\mathcal M$ be selected parameter values. We form a reduced-order basis from the corresponding rows of $\mathbf{\Psi}$ by defining
\begin{equation}\label{eq:psi}
    \psi:\Omega\to\mathbb{R}^r,
    \qquad
    \psi(x):=
    \begin{pmatrix}
    \psi_1(x)\\
    \vdots\\
    \psi_r(x)
    \end{pmatrix},
    \qquad
    \psi_s(\cdot):=\mathbf{\Psi}(\nu_s,\cdot),
    \quad s=1,\ldots,r.
\end{equation}
For each $\mu\in\mathcal M$, we define the best approximation of $\mathbf{\Psi}(\mu,\cdot)$ from the selected row span by least-squares projection. Let
\begin{equation}\label{eq:Gb_cont}
    G:=\int_\Omega \psi(x)\psi(x)^\top\,dx\in\mathbb{R}^{r\times r},
    \qquad
    b(\mu):=\int_\Omega \mathbf{\Psi}(\mu,x)\psi(x)\,dx\in\mathbb{R}^r .
\end{equation}
Assuming that $G$ is invertible, define
\begin{equation}\label{eq:ai_star}
    a^*(\mu):=G^{-1}b(\mu)
    =
    \begin{pmatrix}
    a^*_1(\mu)\\
    \vdots\\
    a^*_r(\mu)
    \end{pmatrix}.
\end{equation}
The resulting rank-$r$ approximation is
\begin{equation}\label{eq:r-approx_star}
    \mathbf{\Psi}^{*r}(\mu,x)
    :=a^*(\mu)^\top\psi(x)
    =
    \sum_{s=1}^{r} a^{*}_{s}(\mu)\psi_s(x).
\end{equation}
Equivalently, for each fixed $\mu\in\mathcal M$, $\mathbf{\Psi}^{*r}(\mu,\cdot)$ is the $L^2(\Omega)$-orthogonal projection of $\mathbf{\Psi}(\mu,\cdot)$ onto
\[
    \operatorname{span}\{\psi_1,\ldots,\psi_r\}.
\]

In practice, $\mathbf{\Psi}$ is observed only at finitely many parameter samples and spatial discretization points. Specifically, suppose that the training parameters $\{\mu_j\}_{j=1}^n$ are drawn independently from $Q$, and that $\{x_k\}_{k=1}^m$ is a regular or quasi-uniform grid on $\Omega$. These samples define the finite matrix
\begin{equation}\label{eq:data_matrix}
    \Psi\in\mathbb{R}^{n\times m},
    \qquad
    \Psi_{j,k} := \mathbf{\Psi}(\mu_j,x_k),
    \qquad
    j=1,\ldots,n,\quad k=1,\ldots,m.
\end{equation}
We select row indices
\[
    I=\{i_1,\ldots,i_r\}\subset\{1,\ldots,n\},
\]
for example by a max-volume row-selection procedure \citep{ALS}, and set
\[
    \nu_s:=\mu_{i_s},
    \qquad s=1,\ldots,r.
\]
Let $\Psi_I\in\mathbb{R}^{r\times m}$ denote the matrix formed from the selected rows of $\Psi$, so that
\[
    (\Psi_I)_{s,k}=\Psi_{i_s,k}=\psi_s(x_k).
\] 

The grid-based Gram matrix and right-hand side are then
\begin{equation}\label{eq:Gmbm}
    G_m
    :=
    \frac1m \Psi_I\Psi_I^\top
    =
    \frac1m\sum_{k=1}^m \psi(x_k)\psi(x_k)^\top
    \in\mathbb{R}^{r\times r},
\end{equation}
and
\begin{equation}
    b_m(\mu)
    :=
    \frac1m\sum_{k=1}^m \mathbf{\Psi}(\mu,x_k)\psi(x_k)
    \in\mathbb{R}^{r}.
\end{equation}
Assuming that $G_m$ is invertible, define the grid-based coefficient map
\begin{equation} \label{eq:a_m}
    a_m:\mathcal{M}\to\mathbb{R}^r,
    \qquad
    a_m(\mu):=G_m^{-1}b_m(\mu).
\end{equation}
This gives the grid-based rank-$r$ approximation
\begin{equation}\label{eq:r-approx_disc}
    \mathbf{\Psi}^{r}_m(\mu,x)
    :=
    a_m(\mu)^\top\psi(x)
    =
    \sum_{s=1}^{r} a_{m,s}(\mu)\psi_s(x).
\end{equation}

We also quantify the coverage of the parameter set $\{\mu_j\}_{j=1}^n$ and the spatial discretization by the corresponding fill distances $h_\mu$ and $h_x$:
\begin{equation}\label{eq:fill_distance}
    h_\mu(n):=\sup_{\mu\in\mathcal M}\min_{1\le j\le n}d_{\mathcal M}(\mu,\mu_j),
    \qquad
    h_x(m):=\sup_{x\in\Omega}\min_{1\le k\le m}\|x-x_k\|_2.
\end{equation}
The discrepancy between the continuous $L^2(\Omega)$ projection $\mathbf{\Psi}^{*r}$ and the grid-based projection $\mathbf{\Psi}^{r}_m$ is controlled by the quadrature assumption in Assumption~\ref{assump3} and by the spatial fill distance.

\paragraph{Parameter-to-coefficient map}
For each sampled parameter $\mu_j$, the coefficient vector $a_m(\mu_j)$ can equivalently be computed by solving the discrete least-squares problem
\begin{equation}\label{eq:lsqa}
    a_m(\mu_j)
    =
    \argmin_{\widetilde a\in\mathbb{R}^r}
    \frac1m
    \sum_{k=1}^m
    \left|
    \mathbf{\Psi}(\mu_j,x_k)
    -
    \widetilde a^\top\psi(x_k)
    \right|^2
    =
    \argmin_{\widetilde a\in\mathbb{R}^r}
    \frac1m
    \left\|
    \Psi_{j,:}
    -
    \widetilde a^\top \Psi_I
    \right\|_2^2 .
\end{equation}
We collect these coefficient vectors in the matrix
\begin{equation}
    A\in\mathbb{R}^{n\times r},
    \qquad
    A_{j,:}=a_m(\mu_j)^\top,
    \qquad
    j=1,\ldots,n.
\end{equation}

To evaluate the reduced-order model at a new parameter value, we approximate the parameter-to-coefficient map $a_m$ by a neural network
\[
    \widehat a:\mathcal{M}\to\mathbb{R}^r.
\]
Specifically, we seek $\widehat a$ in the class
\begin{equation} \label{NN}
\mathcal{N}_r(R,\kappa,L,p,K):=
\left\{S:\mathcal M\to\mathbb R^r \middle |\; 
\begin{aligned}
& S \text{ is representable by a feedforward neural network } \\ 
& \text{ with at most $L$ layers and width at most $p$, } \|S\|_{\infty}\leq R,\; \\
& 
\|W_\ell\|_{\infty,\infty}\leq\kappa,\;
\|b_\ell\|_{\infty}\leq\kappa,\;
\sum_{\ell=1}^L(\|W_\ell\|_0+\|b_\ell\|_0)\leq K    
\end{aligned}
\right\}.
\end{equation}
The neural network is trained by empirical regression against the computed coefficient vectors:
\begin{equation} 
\label{eq:hata}
    \widehat {a}_m
    :=
    \argmin_{\widetilde a\in\mathcal{N}_r(R,\kappa,L,p,K)}
    \frac{1}{n}\sum_{j=1}^n
    \big\|
    \widetilde a(\mu_j)
    -
    a_m(\mu_j)
    \big\|_2^2.
\end{equation}
The final reduced-order approximation of the transport signature is therefore
\begin{equation}\label{emprmin}
    \widehat{\mathbf{\Psi}}_{n,m}^{r}(\mu,x)
    :=
    \widehat a_m(\mu)^\top \psi(x)
    =
    \sum_{s=1}^{r}\widehat{a}_{m,s}(\mu)\,\psi_s(x).
\end{equation}

Our offline stage has two components: extracting the transport signature and building a ROM for it. Working with the transport signature naturally preserves mass and positivity, which gives us the flexibility to use a range of methods for the ROM. For the ROM, skeleton decomposition and a feedforward neural network are not the only possible choices. We adopt skeleton decomposition for its interpretability and a feedforward neural network for its expressive power in learning the parameter-to-coefficient map. Numerically in Section \ref{sec:num}, we observe that selecting the basis via skeleton decomposition yields a parameter-to-coefficient map with better regularity, which the neural network is in turn well suited to approximate.

\subsection{Online stage}
\label{subsec: online}

Given the reduced-order approximation $\widehat{\mathbf{\Psi}}_{n,m}^r$ of the transport signature, we reconstruct an approximation of the target distribution $\rho(\mu)$ by pushing forward the reference distribution $\bar\rho$. Specifically, we define
\begin{equation}\label{eq:rho_lowrank}
    \widehat{\rho}_{n,m}^r(\mu)
    :=
    \bigl(\widehat{\mathbf{T}}_{n,m}^r(\mu,\cdot)\bigr)_{\#}\bar\rho,
\end{equation}
where
\begin{equation}\label{eq:T_lowrank}
    \widehat{\mathbf{T}}_{n,m}^r(\mu,x)
    =
    x-\nabla_x \widehat{\mathbf{\Phi}}_{n,m}^r(\mu,x),
    \qquad
    \widehat{\mathbf{\Phi}}_{n,m}^r
    :=
    (I-\mathcal L)^{-1/2}\widehat{\mathbf{\Psi}}_{n,m}^r .
\end{equation}
Thus, the reduced-order model for the distribution is obtained by first mapping the predicted transport signature back to a transport potential, then differentiating to obtain the transport map, and finally pushing forward the reference distribution.

\begin{longtable}{ 
p{.12\textwidth}  p{.88\textwidth} } 
\toprule
\textbf{Notation} & \textbf{Meaning} \\
\midrule

\multicolumn{2}{@{}l}{\emph{Parameterized distributions and reference measure}} \\
\addlinespace[2pt]

$\mu$ & scalar or vector-valued parameter,  including time when applicable. \\

$Q$ & Probability law of $\mu$ on the parameter space $\mathcal M$. \\

$\mathcal M$ & Compact, connected, $d$-dimensional Riemannian submanifold of $[-b,b]^D$. \\

$\tau_{\mathcal M}$ & Reach of the manifold $\mathcal{M}$. \\

$d_{\mathcal M}$ & Riemannian distance on $\mathcal M$. \\

$\rho(\mu)$ & Probability measure induced by the parameterized PDE at parameter value $\mu$. \\

$\Omega$ & Spatial domain containing the supports of all solution measures, 
$\Omega=[-\xi,\xi]^{D_{\bar\rho}}$. \\

$dx$ & Normalized Lebesgue measure on $\Omega$. \\

$\bar\rho$ & Reference probability distribution on $\Omega$. \\

\midrule
\multicolumn{2}{@{}l}{\emph{Transport potentials, maps, and signatures}} \\
\addlinespace[2pt]

$\phi(\mu)$ & Kantorovich potential for the optimal transport from $\bar\rho$ to $\rho(\mu)$;
$\int_\Omega \phi(\mu)(x)\,dx=0$. \\

$\mathbf{\Phi}(\mu,x)$ & Parameter-dependent Kantorovich potential, with
$\mathbf{\Phi}(\mu,\cdot)=\phi(\mu)$. \\

$\mathbf T(\mu,x)$ & Optimal transport map from $\bar\rho$ to $\rho(\mu)$, given by
$\mathbf T(\mu,\cdot)=I-\nabla_x\mathbf{\Phi}(\mu,\cdot)$. \\

$\mathcal L$ & Weighted Laplacian
$\mathcal L f=\nabla_x\cdot(\bar\rho\nabla_x f)$. \\

$\mathbf{\Psi}(\mu,x)$ & Transport-signature kernel,
$\mathbf{\Psi}=(I-\mathcal L)^{1/2}\mathbf{\Phi}$. \\

$H_\Psi$ & Integral operator with kernel $\mathbf{\Psi}$, $(H_\Psi f)(\mu)=\int_\Omega \mathbf{\Psi}(\mu,x)f(x)\,dx$.
\\
\midrule
\multicolumn{2}{@{}l}{\emph{Samples, grids, and skeleton basis}} \\
\addlinespace[2pt]

$n_{\rm skel}$ & Number of parameter samples used for skeleton selection in the analysis.\\

$\{\mu_\ell^{\rm skel}\}_{\ell=1}^{n_{\rm skel}}$ & The set of parameter samples used to select the skeleton basis. \\
$n$ & Number of parameter samples used to train the coefficient map. \\

$\{\mu_j^{\rm train}\}_{j=1}^{n}$ & The set of parameter samples used for learning the map $\mu\mapsto a_m(\mu)$ in the analysis. These samples are independent to $\{\mu_\ell^{\rm skel}\}_{\ell=1}^{n_{\rm skel}}$.\\

$m$ & Number of spatial grid points. \\

$\{x_k\}_{k=1}^m$ & Spatial discretization points in $\Omega$. \\

$r$ & Target rank of the reduced-order model. \\

$\Psi^{\rm skel}$ & Sampled signature matrix used for skeleton selection, $\Psi^{\rm skel}_{\ell,k}=\mathbf{\Psi}(\mu_\ell^{\rm skel},x_k)$. \\

$\Psi$ & Sampled signature matrix; when no sample splitting is used,
$\Psi_{j,k}=\mathbf{\Psi}(\mu_j,x_k)$. \\

$I$ & Selected row indices from the skeleton decomposition, $I=\{i_1,\ldots,i_r\}$. \\

$J$ & Selected column indices from the skeleton decomposition, $J=\{j_1,\ldots,j_r\}$. \\
$\nu_s$ & Selected parameter value,
$\nu_s=\mu_{i_s}^{\rm skel}$, for $s=1,\ldots,r$. \\

$\zeta_s$ & Selected spatial point,
$\zeta_s=x_{j_s}$, for $s=1,\ldots,r$. \\

$I_\mu$ & Selected parameter set,
$I_\mu=\{\nu_1,\ldots,\nu_r\}\subset\mathcal M$. \\

$J_x$ & Selected spatial point set,
$J_x=\{\zeta_1,\ldots,\zeta_r\}\subset\Omega$. \\

$\psi_s$ & Selected transport-signature basis function,
$\psi_s(\cdot)=\mathbf{\Psi}(\nu_s,\cdot)$. \\

$\psi(x)$ & Vector of selected basis functions,
$\psi(x)=(\psi_1(x),\ldots,\psi_r(x))^\top$. \\

\midrule
\multicolumn{2}{@{}l}{\emph{Projection coefficients and reduced approximations}} \\
\addlinespace[2pt]

$G$ & Continuous Gram matrix, $G=\int_\Omega \psi(x)\psi(x)^\top\,dx$. \\
$b(\mu)$ & Continuous right-hand side, $
b(\mu)=\int_\Omega \mathbf{\Psi}(\mu,x)\psi(x)\,dx$. \\

$a^*(\mu)$ & Continuous projection coefficient vector, $a^*(\mu)=G^{-1}b(\mu)$. \\

$\mathbf{\Psi}^{*r}$ & Continuous $L^2(\Omega)$ projection of $\mathbf{\Psi}$ onto
$\operatorname{span}\{\psi_1,\ldots,\psi_r\}$. \\

$G_m$ & Grid-based Gram matrix, $
G_m=\frac1m\sum_{k=1}^m \psi(x_k)\psi(x_k)^\top$. \\

$b_m(\mu)$ & Grid-based right-hand side, $
b_m(\mu)=\frac1m\sum_{k=1}^m \mathbf{\Psi}(\mu,x_k)\psi(x_k)$. \\

$a_m(\mu)$ & Grid-based coefficient vector,
$a_m(\mu)=G_m^{-1}b_m(\mu)$. \\

$\mathbf{\Psi}_m^r$ & Grid-based rank-$r$ approximation, $
\mathbf{\Psi}_m^r(\mu,x)=a_m(\mu)^\top\psi(x)$. \\

$\widehat a_m$ & Learned approximation of the grid-based coefficient map $a_m$. \\

$\widehat{\mathbf{\Psi}}_{n,m}^r$ & Learned rank-$r$ approximation of $\mathbf{\Psi}$, $
\widehat{\mathbf{\Psi}}_{n,m}^r(\mu,x)
=
\widehat a_m(\mu)^\top\psi(x)$. \\

$\widehat{\mathbf{\Phi}}_{n,m}^r$ & Reconstructed transport potential, $\widehat{\mathbf{\Phi}}_{n,m}^r
=
(I-\mathcal L)^{-1/2}\widehat{\mathbf{\Psi}}_{n,m}^r$. \\

$\widehat{\mathbf T}_{n,m}^r$ & Approximated transport map, $
\widehat{\mathbf T}_{n,m}^r(\mu,x)
= x-\nabla_x\widehat{\mathbf{\Phi}}_{n,m}^r(\mu,x)$. \\

$\widehat\rho_{n,m}^r(\mu)$ & Reconstructed ROM distribution, $\widehat\rho_{n,m}^r(\mu)=\bigl(\widehat{\mathbf T}_{n,m}^r(\mu,\cdot)\bigr)_\#\bar\rho$. \\

\midrule
\multicolumn{2}{@{}l}{\emph{Quantities used in the error analysis}} \\
\addlinespace[2pt]

$h_\mu^{\rm skel}(n_{\rm skel})$ & Fill distance of the skeleton-selection parameter sample\\
$h_x(m)$ & Fill distance of the spatial grid, $
h_x(m)=\sup_{x\in\Omega}\min_{1\le k\le m}\|x-x_k\|_2$.\\

$\Delta_r^{\rm cont}$ & Maximal continuous $r\times r$ volume of the kernel $\mathbf{\Psi}$. \\

$\theta$ & Fraction of the maximal continuous volume captured by the selected skeleton.\\

$\gamma_r,\Gamma_r$ & Minimum and maximum eigenvalues of the continuous Gram matrix $G$. \\

$\sigma_k(\mathbf{\Psi})$ & $k$th singular value of the integral operator $H_\Psi$. \\

\bottomrule
\caption{Notation used in the OT-based reduced-order modeling framework.}
\label{tab:ot-rom-notation}
\end{longtable}

\section{Theoretical Analysis}
\label{sec:theory}
In this section, we present our main result, Theorem~\ref{mainthm}, which establishes a mean-squared $W_2$ error bound between the reconstructed distribution $\widehat{\rho}_{n,m}^r(\mu)$ and the true distribution $\rho(\mu)$. We begin by stating a few necessary assumptions.

\subsection{Assumptions}
We first specify the sample-splitting convention used in the analysis. Let \[ \{\mu_\ell^{\rm skel}\}_{\ell=1}^{n_{\rm skel}} \subset \mathcal M \] be the sampled parameter set  used for skeleton selection, and let $\{x_k\}_{k=1}^m\subset\Omega$ be the spatial grid. These samples define the skeleton-selection matrix \[ \Psi^{\rm skel}\in\mathbb R^{n_{\rm skel}\times m}, \qquad \Psi^{\rm skel}_{\ell,k} := \mathbf{\Psi}(\mu_\ell^{\rm skel},x_k). \] A maximum-volume skeleton decomposition selects row indices 
\[ 
I=\{i_1,\ldots,i_r\}\subset\{1,\ldots,n_{\rm skel}\} \] 
and column indices 
\[ J=\{j_1,\ldots,j_r\}\subset\{1,\ldots,m\}. 
\] 
We form the reduced basis in \eqref{eq:psi} by selected parameter and spatial points \[ \nu_s:=\mu_{i_s}^{\rm skel}, \qquad \zeta_s:=x_{j_s}, \qquad s=1,\ldots,r. \]

Conditional on this selected basis and spatial grid, let \[ \{\mu_j^{\rm train}\}_{j=1}^{n} \] be an independent parameter sample set drawn from $Q$. This second sample set is used only to train the parameter-to-coefficient map. The continuous projection coefficients are denoted by $a^*(\mu)$, while the grid-based coefficients are denoted by $a_m(\mu)$. Accordingly, $\mathbf{\Psi}^{*r}$ denotes the continuous $L^2(\Omega)$ projection, $\mathbf{\Psi}_m^r$ denotes its grid-based counterpart, and $\widehat{\mathbf{\Psi}}_{n,m}^r$ denotes the learned reduced-order model. The dependence on the skeleton-selection sample is suppressed in the notation, since the statistical estimates are conditional on the selected basis and grid.

We now state the assumptions used to prove the main error bound. The first assumption is geometric and ensures that the parameter space can be treated as a well-behaved compact manifold for sampling and learning estimates.

\begin{assumption} \label{assump1}
We assume that $\mathcal M$ has positive reach
$\tau_{\mathcal M}>0$, where
\[
\tau_{\mathcal M}
:=
\sup\{r>0:\text{ every } x \text{ in } \mathbb{R}^D\text{ with }{\rm dist}(x,\mathcal M)<r
\text{ has a unique nearest point on }\mathcal M\}.
\]
\end{assumption}

Our second assumption concerns the regularity of $\mathbf{\Psi}$.

\begin{assumption}[Regularity of $\mathbf{\Psi}$]\label{assump2}
Suppose $\mathbf{\Psi}:\mathcal{M}\times\Omega\to\mathbb{R}$ satisfies the
following conditions.
\begin{itemize}
\item $\mathbf{\Psi}$ is nondegenerate at rank $r$, in the sense that
\begin{equation} \label{Delta_cont}
\Delta_r^{\rm cont}
:=
\max_{\substack{(\mu_1,\ldots,\mu_r)\in\mathcal M^r\\
(x_1,\ldots,x_r)\in\Omega^r}}
\left|
\det\big[\mathbf{\Psi}(\mu_i,x_j)\big]_{i,j=1}^r
\right|
>0.
\end{equation}

\item $\mathbf{\Psi}$ is Lipschitz continuous with constant
$L_{\mathbf{\Psi}}>0$:
\[
|\mathbf{\Psi}(\mu',x')-\mathbf{\Psi}(\mu,x)|
\le
L_{\mathbf{\Psi}}
\sqrt{d_{\mathcal{M}}(\mu,\mu')^2+\|x'-x\|_2^2},
\]
where $d_{\mathcal{M}}$ is the Riemannian distance on $\mathcal{M}$.

\item $\mathbf{\Psi}$ is uniformly bounded:
\[
\sup_{(\mu,x)\in\mathcal M\times\Omega}
|\mathbf{\Psi}(\mu,x)|\le U
\]
for some constant $U>0$.
\end{itemize}
\end{assumption}

The next assumption concerns the reduced basis and the spatial quadrature rule. It ensures that the selected signatures are linearly independent in $L^2(\Omega)$ and that the grid-based inner products used to compute $a_m(\mu)$ are accurate approximations of the corresponding continuous inner products.
\begin{assumption}[Reduced basis and quadrature]\label{assump3}
Let
\[
\psi(x):=(\psi_1(x),\ldots,\psi_r(x))^\top,
\qquad
G:=\int_\Omega \psi(x)\psi(x)^\top\,dx,
\qquad
G_m:=\frac1m\sum_{k=1}^m \psi(x_k)\psi(x_k)^\top.
\]
Assume that:
\begin{enumerate}
    \item[(i)] The Gram matrix $G$ is positive definite, with
    \[
    \gamma_r:=\lambda_{\min}(G)>0,
    \qquad
    \Gamma_r:=\lambda_{\max}(G)<\infty.
    \]

    \item[(ii)] The quadrature rule satisfies
    \[
    \left|
    \int_\Omega f(x)\,dx-\frac1m\sum_{k=1}^m f(x_k)
    \right|
    \le C_{\rm quad}\,\mathrm{Lip}(f)\,h_x(m)
    \]
    for every Lipschitz function $f$.

    \item[(iii)] The fill distance is sufficiently small:
    \[
    h_x(m)\le
    \frac{\gamma_r}{4C_{\rm quad}rUL_{\mathbf\Psi}},
    \]
    where $U$ and $L_{\mathbf\Psi}$ are defined in
    Assumption~\ref{assump2}.
\end{enumerate}
\end{assumption}

The final assumption isolates the low-rank approximation property supplied by the skeleton decomposition. It states that if the selected rows and columns capture a fixed fraction of the maximal continuous volume, then the associated skeleton approximation has an error controlled by the next singular value of the transport-signature kernel.

The column set $J_x$ is used only in the analysis of the skeleton approximation; the reduced-order basis itself is formed from the selected rows
$\{\psi_s\}_{s=1}^r$.

\begin{assumption}[Low-rank approximation]\label{assump4} 
Consider the kernel function $\mathbf{\Psi}:\mathcal M\times \Omega\to \mathbb R$. Let
\[
    I_\mu=\{\nu_1,\ldots,\nu_r\}\subset\mathcal M,
    \qquad
    J_x=\{\zeta_1,\ldots,\zeta_r\}\subset\Omega
\]
be selected parameter and spatial point sets. Suppose that
\begin{equation}\label{eq:volume_condition_thm}
    \left|\det \mathbf{\Psi}(I_\mu,J_x)\right|
    \ge
    \theta\,\Delta_r^{\rm cont}\,,\quad \text{for some $\theta\in(0,1]$}\,,
\end{equation}
where $\Delta_r^{\rm cont}$ is defined in~\eqref{Delta_cont}. Define the continuous skeleton approximation
\begin{equation}\label{eq:psi_ske}
    \mathbf{\Psi}_r^{\rm skel}(\mu,x)
    :=
    \mathbf{\Psi}(\mu,J_x)
    [\mathbf{\Psi}(I_\mu,J_x)]^{-1}
    \mathbf{\Psi}(I_\mu,x).
\end{equation}
Assume that
\[
    \left\|
    \mathbf{\Psi}-\mathbf{\Psi}_r^{\rm skel}
    \right\|_{L^\infty(\mathcal M\times\Omega)}
    \leq
    \frac{C (r+1)}{\theta}\,
    \sigma_{r+1}(\mathbf{\Psi}),
\]
where $\sigma_{r+1}(\mathbf{\Psi})$ denotes the $(r+1)$-th singular value of the operator $H_\Psi$ in~\eqref{eq:H}. The constant $C$ is independent of $\theta$, $r$, and the particular kernel $\mathbf{\Psi}$, but may depend on the regularity class under consideration.
\end{assumption}

\begin{remark}[Motivation for Assumption~\ref{assump4}]
Assumption~\ref{assump4} is the continuous-kernel analogue of the classical max-volume skeleton approximation bound for finite matrices. Indeed, it was established in \cite{ALS} that, if $A\in\mathbb R^{m\times n}$ and $A_{I,J}$ has volume at least a fraction $\theta\in(0,1]$ of the maximum $r\times r$ volume, then the skeleton approximation
\[
    A_{:,J}A_{I,J}^{-1}A_{I,:}
\]
satisfies a Chebyshev-norm estimate of the form
\[
    \left\|A-A_{:,J}A_{I,J}^{-1}A_{I,:}\right\|_C
    \le
    \frac{r+1}{\theta\sqrt{\sum_{k=1}^{r+1}\sigma_k^{-2}(A)}}
    \le
    \frac{r+1}{\theta}\sigma_{r+1}(A).
\]
Assumption~\ref{assump4} postulates the corresponding bound for the continuous matrix $\mathbf{\Psi}(\mu,x)$, with the matrix maximum norm replaced by $L^\infty(\mathcal M\times\Omega)$ and the matrix singular values replaced by the singular values of the integral operator $H_\Psi$ in~\eqref{eq:H}.
\end{remark}

\subsection{Main Result}

Before stating the main theorem, we record two auxiliary facts. The first transfers error estimates for transport maps into $W_2$ error estimates for the corresponding pushforward measures. The second relates the maximal volume of a sampled matrix to the maximal volume of the underlying continuous kernel; this justifies using a discrete skeleton decomposition to select representative transport signatures.

\begin{lemma} \label{lemma:OT_distance}
Let $\lambda$ be a probability measure on a compact domain
$\Omega\subseteq\mathbb{R}^{D_{\bar\rho}}$. Let
$\rho_1:=(T_1)_{\#}\lambda$ and
$\rho_2:=(T_2)_{\#}\lambda$, where
$T_1,T_2\in L^2_\lambda(\mathbb{R}^{D_{\bar\rho}};\mathbb{R}^{D_{\bar\rho}})$. Then
\begin{equation}
    W^2_2(\rho_1,\rho_2)
    \leq
    \|T_1-T_2\|_{L^2_\lambda(\mathbb{R}^{D_{\bar\rho}};\mathbb{R}^{D_{\bar\rho}})}^2.
\end{equation}
\end{lemma}
\begin{proof}
    Let us define a coupling map $\pi$ between $\rho_1$ and $\rho_2$ by pushing $\lambda$ forward through the map $(T_1,T_2)$, i.e., $\pi:=(T_1,T_2)_{\#}\lambda$. Then $\pi$ has the first marginal $(T_1)_{\#}\lambda=\rho_1$ and second marginal $(T_2)_{\#}\lambda=\rho_2$, so $\pi\in\Pi(\rho_1,\rho_2)$.
    By the definition of 2-Wasserstein distance,
    \begin{equation*}
        W^2_2(\rho_1,\rho_2)=\inf_{\gamma\in\Pi(\rho_1,\rho_2)}\int|y-z|^2 d\gamma(y,z)\leq \int|y-z|^2d\pi(y,z).
    \end{equation*}
    Now compute the cost under $\pi$ using the pushforward definition,
    \begin{equation*}
        \int|y-z|^2d\pi(y,z)=\int_{\Omega}|T_1(x)-T_2(x)|^2 d\lambda(x)=\|T_1-T_2\|_{L^2_\lambda(\mathbb{R}^{D_{\bar\rho}};\mathbb{R}^{D_{\bar\rho}})}^2.
    \end{equation*}
    Combining the two inequalities gives the desired result.
\end{proof}

\begin{lemma}\label{lem:nu_bound_mult}
Let Assumption~\ref{assump2} hold and set $\Omega=[-\xi,\xi]^{D_{\bar\rho}}$.
Let $\{\mu_j\}_{j=1}^n\subset\mathcal M$ and $\{x_k\}_{k=1}^m\subset\Omega$ be sampling sets and
define the sampled matrix $\Psi\in\mathbb R^{n\times m}$ by~\eqref{eq:data_matrix}. Define $\varepsilon_{n,m}:=L_{\mathbf{\Psi}}\big(h_\mu(n)+h_x(m)\big)$.

For $r\ge1$, define the continuous and discrete maximal volumes 
\[
\Delta_r^{\mathrm{cont}}
:=
\sup_{\substack{I_\mu\subset\mathcal{M},\,J_x\subset\Omega\\ |I_\mu|=|J_x|=r}}
|\det\mathbf{\Psi}(I_\mu,J_x)|,
\qquad
\Delta_r^{\mathrm{disc}}
:=
\max_{\substack{I\subset\{1,\ldots,n\},\,J\subset\{1,\ldots,m\}\\ |I|=|J|=r}}
|\det\Psi(I,J)|.
\]
Let $C_r:=r^{\frac r2+1}$. Then
\[
\Delta_r^{\mathrm{disc}}
\;\ge\;
\Delta_r^{\mathrm{cont}}-C_r\,U^{r-1}\,\varepsilon_{n,m}.
\]
Consequently, defining
\[
\theta_{n,m}
:=
\max\!\left\{0,\,
1-\frac{C_r\,U^{r-1}\,L_{\mathbf{\Psi}}\big(h_\mu(n)+h_x(m)\big)}{\Delta_r^{\mathrm{cont}}}
\right\}\in[0,1],
\]
we obtain
\[
\Delta_r^{\mathrm{disc}}
\ge
\theta_{n,m}\,\Delta_r^{\mathrm{cont}}.
\]
In particular, if
\[
h_\mu(n)+h_x(m)
<
\frac{\Delta_r^{\mathrm{cont}}}
{C_r\,U^{r-1}\,L_{\mathbf{\Psi}}},
\]
then $\theta_{n,m}\in(0,1]$.
\end{lemma}

\begin{proof}
By Assumption~\ref{assump2}, $\mathbf{\Psi}$ is Lipschitz and therefore continuous on the compact set
$\mathcal M\times\Omega$. Hence the map
\[
(\mu_1,\ldots,\mu_r,x_1,\ldots,x_r)
\mapsto
\det\big(\mathbf{\Psi}(\mu_i,x_j)\big)_{i,j=1}^r
\]
is continuous on the compact set $(\mathcal M^r)\times(\Omega^r)$, and therefore the supremum
defining $\Delta_r^{\mathrm{cont}}$ is attained. Thus there exist
$I^*=\{\mu_1^*,\ldots,\mu_r^*\}\subset\mathcal M$ and
$J^*=\{x_1^*,\ldots,x_r^*\}\subset\Omega$
such that
\[
\Delta_r^{\mathrm{cont}}
=
\big|\det A\big|,
\qquad
A:=\big(\mathbf{\Psi}(\mu_i^*,x_j^*)\big)_{i,j=1}^r .
\]
For sufficiently small fill distances, these sampled representatives may be chosen to be distinct, since the maximizing tuples have nonzero determinant and hence consist of distinct parameter and spatial points.

By definition of the fill distances $h_\mu(n)$ and $h_x(m)$, for each $\mu_i^*$ there exists
$\widehat i_i\in\{1,\ldots,n\}$ such that
\[
d_{\mathcal M}(\mu_i^*,\mu_{\widehat i_i})\le h_\mu(n),
\]
and for each $x_j^*$ there exists $\widehat j_j\in\{1,\ldots,m\}$ such that
\[
\|x_j^*-x_{\widehat j_j}\|_2\le h_x(m).
\]
Let $\widehat I=\{\widehat i_1,\ldots,\widehat i_r\}$ and $\widehat J=\{\widehat j_1,\ldots,\widehat j_r\}$, and define 
\[
\widehat A
:=
\Psi(\widehat I,\widehat J)
=
\big(\mathbf{\Psi}(\mu_{\widehat i_i},x_{\widehat j_j})\big)_{i,j=1}^r .
\]

By the Lipschitz continuity of $\mathbf{\Psi}$,
\[
|\widehat A_{ij}-A_{ij}|
=
|\mathbf{\Psi}(\mu_{\widehat i_i},x_{\widehat j_j})-\mathbf{\Psi}(\mu_i^*,x_j^*)|
\le
L_{\mathbf{\Psi}}(h_\mu(n)+h_x(m))
=
\varepsilon_{n,m}.
\]
Hence $\widehat A=A+E$ with $|E_{ij}|\le\varepsilon_{n,m}$.

Let $a_\ell$ and $\widehat a_\ell$ be the $\ell$th column of $A$ and
$\widehat A=A+E$, respectively, with $e_\ell=\widehat a_\ell-a_\ell$.  
Using a telescoping expansion, one has
\begin{align}
|\det(A+E)-\det(A)|&=|\sum_{\ell=1}^r\det(\widehat a_1,\ldots,\widehat a_{\ell-1},
\underbrace{\widehat a_\ell- a_\ell}_{=e_\ell},a_{\ell+1},\ldots,a_r)|\\
&\le
\sum_{\ell=1}^r
\left|
\det(\widehat a_1,\ldots,\widehat a_{\ell-1},
e_\ell,a_{\ell+1},\ldots,a_r)
\right|.\notag
\end{align}
Since both $A$ and $\widehat A$ have entries bounded by $U$, their columns
have Euclidean norm at most $\sqrt r\,U$, while
$\|e_\ell\|_2\le \sqrt r\,\varepsilon_{n,m}$. Hence by Hadamard's inequality,
\[
|\det(A+E)-\det(A)|
\le
r(\sqrt r\,\varepsilon_{n,m})(\sqrt r\,U)^{r-1}
=
C_rU^{r-1}\varepsilon_{n,m},
\qquad
C_r=r^{\frac r2+1}.
\]

Consequently,
\[
|\det \widehat A|
\ge
|\det A|-|\det(\widehat A)-\det(A)|
\ge
\Delta_r^{\mathrm{cont}}-C_r\,U^{r-1}\varepsilon_{n,m}.
\]
Since $\Delta_r^{\mathrm{disc}}$ is the maximal determinant over all sampled
$r\times r$ submatrices, it follows that
\[
\Delta_r^{\mathrm{disc}}
\ge
|\det\Psi(\widehat I,\widehat J)|
=
|\det\widehat A|
\ge
\Delta_r^{\mathrm{cont}}-C_r\,U^{r-1}\varepsilon_{n,m}.
\]

Dividing by $\Delta_r^{\mathrm{cont}}$ and defining
\[
\theta_{n,m}
=
\max\!\left\{
0,\,
1-\frac{C_r\,U^{r-1}L_{\mathbf{\Psi}}(h_\mu(n)+h_x(m))}
{\Delta_r^{\mathrm{cont}}}
\right\},
\]
yields
\[
\Delta_r^{\mathrm{disc}}
\ge
\theta_{n,m} \Delta_r^{\mathrm{cont}}.
\]
\end{proof}

We now state the main result, which establishes a mean-squared $W_2$ generalization error bound between the reconstructed distribution $\widehat{\rho}_{n,m}^r(\mu)$ and the true distribution $\rho(\mu)$.
Note that we work under the sample-splitting convention. To be precise, the selected basis $I_\mu$ and spatial grid $J_x$ are fixed, and the randomness of the ROM $\hat{\rho}^r_{n,m}$ comes from the independent training samples $\{\mu_j^{\rm train}\}_{j=1}^n$ used to learn the empirical coefficient map $\hat{a}_m$.
\begin{theorem}\label{mainthm}
Suppose Assumptions~\ref{assump1}--\ref{assump4} hold under the
sample-splitting convention described above.  
Conditional on the selected basis $I_\mu$ and grid $J_x$, let $\widehat a_m$ be the empirical risk minimizer in~\eqref{eq:hata} over the network class $\mathcal N_r(R,\kappa,L,p,K)$ defined in~\eqref{NN}. Choose $R\ge R_a$ where $R_a:=\frac{2\sqrt r\,U^2}{\gamma_r}$, and choose the remaining network parameters as
\begin{align*}
L =\tilde{O}\left(\frac{\log n}{2+d}\right), \quad 
p=\tilde{O}\left(rn^{\frac{d}{2+d}}\right),\quad K=\tilde{O}\left(\frac{r}{2+d}n^{\frac{d}{2+d}}\log n\right),
\quad
\kappa=O(\max\{1,b,\sqrt{d},\tau^2_{\mathcal{M}}\}).
\end{align*}
Let $\sigma_k(\mathbf{\Psi})$ denote the $k$th singular value of the integral operator $H_\Psi$ in~\eqref{eq:H}. Then
\begin{align} 
\nonumber
&\mathbb{E}\left[\int_{\mathcal{M}}
W^2_2\left(\rho(\mu), \widehat{\rho}^r_{n,m}(\mu)\right)dQ(\mu)\right] \\
&\leq
\frac{4C^2(r+1)^2\sigma^2_{r+1}(\mathbf{\Psi})}{\theta^2}  
+ 4C_1\,\Gamma_r\,h_x(m)^2 
+2c\Gamma_r \frac{r^2\,U^4}{\gamma_r^2}
\left(n^{-\frac{2}{2+d}}+\frac Dn\right)\log^3n.
\label{eq:totalerror}
\end{align}
Here, the expectation is taken over the training samples, $C_1$ is the constant that depends only on
$r,\gamma_r,U,L_{\mathbf{\Psi}},C_{\rm quad}$, and $c$
is a constant depending on $\log D$, $d$, $b$, $\kappa$,
$\mathcal M$, and the Lipschitz radius of the coefficient functions. The notation $\tilde{O}(\cdot)$ hides factors depending on $d$ and some logarithmic
factors.

Moreover, if the skeleton-selection parameter sample and spatial grid have fill distances
\[
    h_\mu^{\rm skel}(n_{\rm skel})
    :=
    \sup_{\mu\in\mathcal M}
    \min_{1\le \ell\le n_{\rm skel}}
    d_{\mathcal M}(\mu,\mu_\ell^{\rm skel})
\]
and $h_x(m)$, then the volume condition~\eqref{eq:volume_condition_thm} holds with
\[
\theta
=
1-
\frac{C_{\rm det}(r)U^{r-1}L_{\mathbf\Psi}
\big(h_\mu^{\rm skel}(n_{\rm skel})+h_x(m)\big)}
{\Delta_r^{\rm cont}},
\]
provided the right-hand side is positive.
\end{theorem}

The proof decomposes the total error into three terms. The first is the deterministic low-rank approximation error from representing $\mathbf{\Psi}$ by the span of the selected signatures. The second is the discretization error from replacing continuous $L^2(\Omega)$ projections by grid-based least-squares projections. The third is the statistical error from learning the parameter-to-coefficient map $a_m$ from finitely many training samples. We estimate these three contributions separately and then combine them in the proof of Theorem~\ref{mainthm}.

\paragraph{$\bullet$ Low-rank approximation error.}
We first control the error from approximating the transport-signature kernel $\mathbf{\Psi}$ by the span of the selected signatures $\{\psi_s\}_{s=1}^r$. Assumption~\ref{assump4} gives a skeleton approximation bound in $L^\infty(\mathcal M\times\Omega)$; the following lemma transfers this bound to the continuous $L^2(\Omega)$ projection $\mathbf{\Psi}^{*r}$ used in the ROM construction.

\begin{lemma}[Low-rank error]\label{lemma:low_rank_error}
Let $\mathbf{\Psi}^{*r}$ be the continuous $L^2(\Omega)$ projection defined in~\eqref{eq:r-approx_star}. Suppose Assumptions~\ref{assump1}--\ref{assump4} hold. Then
\begin{equation}
    \left\|
    \mathbf{\Psi}-\mathbf{\Psi}^{*r}
    \right\|_{L^2_Q(\mathcal{M};L^2(\Omega))}
    \leq
    \frac{C(r+1)}{\theta}
    \sigma_{r+1}(\mathbf{\Psi}),
\end{equation}
where $C$ is the constant from Assumption~\ref{assump4}.
\end{lemma}
\begin{proof}
By definition, for each fixed $\mu$, $\mathbf{\Psi}^{*r}(\mu,\cdot)$ is the
$L^2(\Omega)$-orthogonal projection of $\mathbf{\Psi}(\mu,\cdot)$ onto
$\operatorname{span}\{\psi_1,\ldots,\psi_r\}$. Therefore, for any admissible
coefficient map $\widetilde a:\mathcal M\to\mathbb R^r$,
\[
\left\|
\mathbf{\Psi}-\mathbf{\Psi}^{*r}
\right\|_{L^2_Q(\mathcal{M};L^2(\Omega))}
\le
\left\|
\mathbf{\Psi}-\widetilde a(\cdot)^\top\psi(\cdot)
\right\|_{L^2_Q(\mathcal{M};L^2(\Omega))}.
\]
Taking
\[
\widetilde a(\mu)^\top
=
\mathbf{\Psi}(\mu,J_x)
[\mathbf{\Psi}(I_\mu,J_x)]^{-1}
\]
gives
\[
\left\|
\mathbf{\Psi}-\mathbf{\Psi}^{*r}
\right\|_{L^2_Q(\mathcal{M};L^2(\Omega))}
\le
\left\|
\mathbf{\Psi}-\mathbf{\Psi}_r^{\rm skel}
\right\|_{L^2_Q(\mathcal{M};L^2(\Omega))}.
\]
Since $Q$ and $dx$ are probability measures,
\[
\left\|
\mathbf{\Psi}-\mathbf{\Psi}_r^{\rm skel}
\right\|_{L^2_Q(\mathcal{M};L^2(\Omega))}
\le
\left\|
\mathbf{\Psi}-\mathbf{\Psi}_r^{\rm skel}
\right\|_{L^\infty(\mathcal M\times\Omega)}.
\]
The result now follows from Assumption~\ref{assump4}.
\end{proof}

\paragraph{$\bullet$ Discretization error.}

We next compare the ideal continuous projection $\mathbf{\Psi}^{*r}$ with the grid-based projection $\mathbf{\Psi}_m^r$ used by the algorithm. This is the only place where the spatial quadrature error enters the deterministic part of the analysis.

\begin{lemma}[Discretization error]\label{lemma:disc_error}
Suppose Assumptions~\ref{assump1}--\ref{assump3} hold. Then
\begin{equation}
\|\mathbf{\Psi}^{*r}-\mathbf{\Psi}_m^r\|_{L^2_Q(\mathcal{M};L^2(\Omega))}^2
\le
C_1\,\Gamma_r\,h_x(m)^2,
\end{equation}
where $C_1$ depends only on
$r,\gamma_r,U,L_{\mathbf{\Psi}},C_{\rm quad}$.
\end{lemma}

\begin{proof}
Notice that
\[
\mathbf{\Psi}^{*r}(\mu,\cdot)-\mathbf{\Psi}_m^r(\mu,\cdot)
=
\sum_{s=1}^r
\big(a^*_{s}(\mu)-a_{m,s}(\mu)\big)\psi_s(\cdot).
\]
Therefore,
\[
\|\mathbf{\Psi}^{*r}-\mathbf{\Psi}_m^{r}\|_{L^2_Q(\mathcal{M};L^2(\Omega))}^2
=
\int_{\mathcal M} \big(a^*(\mu)-a_m(\mu)\big)^\top G
\big(a^*(\mu)-a_m(\mu)\big)\,dQ(\mu)
\le
\Gamma_r \|a^*-a_m\|_{L^2_Q(\mathcal{M};\mathbb R^r)}^2.
\]
Thus it remains to estimate $\|a^*-a_m\|_{L^2_Q(\mathcal{M};\mathbb R^r)}$.

We first control $G_m-G$. For each $i,j\in\{1,\ldots,r\}$, the function
$x\mapsto \psi_i(x)\psi_j(x)$ is Lipschitz and satisfies
\[
\mathrm{Lip}(\psi_i\psi_j)
\le
\|\psi_i\|_{L^\infty(\Omega)}\mathrm{Lip}(\psi_j)
+
\|\psi_j\|_{L^\infty(\Omega)}\mathrm{Lip}(\psi_i)
\le 2UL_{\mathbf{\Psi}}.
\]
Hence, by Assumption~\ref{assump3}  (ii),
\[
\big|(G_m-G)_{ij}\big|
\le
2C_{\rm quad}UL_{\mathbf{\Psi}}\,h_x(m).
\]
Therefore,
\[
\|G_m-G\|_2
\le
r\max_{i,j}|(G_m-G)_{ij}|
\le
2C_{\rm quad}rUL_{\mathbf{\Psi}}\,h_x(m).
\]
Since Assumption~\ref{assump3}  (iii) gives
\[
h_x(m)\le \frac{\gamma_r}{4C_{\rm quad}rUL_{\mathbf{\Psi}}},
\]
we have
\[
\|G_m-G\|_2\le \frac{\gamma_r}{2}.
\]
By Weyl's inequality,
\[
\lambda_{\min}(G_m)\ge \frac{\gamma_r}{2},
\qquad
\|G_m^{-1}\|_2\le \frac{2}{\gamma_r}.
\]

Next we control $b_m-b$. For each $i\in\{1,\ldots,r\}$ and fixed
$\mu\in\mathcal M$, the function
\[
x\mapsto \mathbf{\Psi}(\mu,x)\psi_i(x)
\]
is Lipschitz with constant at most $2UL_{\mathbf{\Psi}}$. Hence
\[
\left|(b_m(\mu)-b(\mu))_i\right|
\le
2C_{\rm quad}UL_{\mathbf{\Psi}}\,h_x(m),
\]
and therefore
\[
\|b_m-b\|_{L^2_Q(\mathcal{M};\mathbb R^r)}
\le
2\sqrt r\,C_{\rm quad}UL_{\mathbf{\Psi}}\,h_x(m).
\]

By definition,
\[
a^*(\mu)=G^{-1}b(\mu),
\qquad
a_m(\mu)=G_m^{-1}b_m(\mu).
\]
Thus, for $Q$-a.e.\ $\mu$,
\[
a^*(\mu)-a_m(\mu)
=
G_m^{-1}(b(\mu)-b_m(\mu))
+
G_m^{-1}(G_m-G)a^*(\mu).
\]
Taking $L^2_Q(\mathcal{M};\mathbb R^r)$ norms gives
\[
\|a^*-a_m\|_{L^2_Q(\mathcal{M};\mathbb R^r)}
\le
\|G_m^{-1}\|_2
\Big(
\|b-b_m\|_{L^2_Q(\mathcal{M};\mathbb R^r)}
+
\|G_m-G\|_2\,\|a^*\|_{L^2_Q(\mathcal{M};\mathbb R^r)}
\Big).
\]
Moreover, since
\[
\|b(\mu)\|_2\le \sqrt r\,U^2,
\]
we have
\[
\|a^*(\mu)\|_2
\le
\frac{\sqrt r\,U^2}{\gamma_r}
=:R_*.
\]
Combining the previous bounds yields
\[
\|a^*-a_m\|_{L^2_Q(\mathcal{M};\mathbb R^r)}
\le
\frac{4C_{\rm quad}UL_{\mathbf\Psi}\sqrt r}{\gamma_r}
\left(1+rR_*\right)h_x(m),
\]
which finishes the proof.
\end{proof}

\paragraph{$\bullet$ Statistical error.}
Finally, we estimate the error caused by learning the coefficient map $a_m:\mathcal M\to\mathbb R^r$ from finitely many parameter samples. 
Recall that we adopt a sample-splitting convention: we first fix the skeleton basis and grid, and then draw a training sample set to learn $a_m$. Thus, for this part of the analysis, the only randomness in the estimates below arises from the training sample $\{\mu_j^{\rm train}\}_{j=1}^n$, which is independent of the fixed basis and grid.

The mean-squared statistical error satisfies the following lemma, whose proof uses \citep[Theorem~2]{chen2022nonparametric}.

\begin{lemma}[Statistical error]\label{propGen}
Suppose Assumptions~\ref{assump1}--\ref{assump3} hold. Conditional on the selected basis and spatial grid, let $\{\mu_j^{\rm train}\}_{j=1}^n$ be i.i.d.\ samples from $Q$. Let $\widehat a_m$ be the empirical risk minimizer in~\eqref{eq:hata} over the feedforward neural network class $\mathcal{N}_r(R,\kappa,L,p,K)$ defined in~\eqref{NN}, with parameters $L,p,K,R,\kappa$ chosen as stated in Theorem~\ref{mainthm}.
Then
\begin{equation}
\mathbb{E}\!\left[
\|\mathbf{\Psi}_m^r-\widehat{\mathbf{\Psi}}_{n,m}^r\|_{L^2_Q(\mathcal{M};L^2(\Omega))}^2
\right]
\leq
c\Gamma_r \frac{r^2\,U^4}{\gamma_r^2}
\left(n^{-\frac{2}{2+d}}+\frac{D}{n}\right)\log^3n,
\end{equation}
where the expectation is taken over the training samples and $c$ is a constant depending on $\log D$, $d$, $b$, $\kappa$, $\mathcal{M}$, and the coefficient Lipschitz radius.
\end{lemma}

\begin{proof}[Proof of Lemma~\ref{propGen}]
Since
\[
\mathbf{\Psi}_m^r(\mu,\cdot)=\sum_{s=1}^{r} a_{m,s}(\mu)\psi_s(\cdot),
\qquad
\widehat{\mathbf{\Psi}}_{n,m}^r(\mu,\cdot)
=
\sum_{s=1}^{r} \widehat{a}_{m,s}(\mu)\psi_s(\cdot),
\]
we have
\[
\mathbf{\Psi}_m^r(\mu,\cdot)-\widehat{\mathbf{\Psi}}_{n,m}^r(\mu,\cdot)
=
\sum_{s=1}^{r}
\big(a_{m,s}(\mu)-\widehat a_{m,s}(\mu)\big)\psi_s(\cdot).
\]
Let
\[
e_a(\mu):=
(a_{m,1}(\mu)-\widehat a_{m,1}(\mu),\ldots,
a_{m,r}(\mu)-\widehat a_{m,r}(\mu))^\top.
\]
Then, for each fixed $\mu$,
\[
\left\|
\sum_{s=1}^r
\big(a_{m,s}(\mu)-\widehat a_{m,s}(\mu)\big)\psi_s(\cdot)
\right\|^2_{L^2(\Omega)}
=
e_a(\mu)^\top G\, e_a(\mu)
\le
\Gamma_r\|e_a(\mu)\|_2^2.
\]
Hence
\begin{align*}
\mathbb{E}\!\left[
\|\mathbf{\Psi}_m^r-\widehat{\mathbf{\Psi}}_{n,m}^r\|_{L^2_Q(\mathcal{M};L^2(\Omega))}^2
\right]
&\le
\Gamma_r
\sum_{s=1}^{r}
\mathbb{E}\left[
\int_{\mathcal M}
\big(a_{m,s}(\mu)-\widehat a_{m,s}(\mu)\big)^2\,dQ(\mu)
\right].
\end{align*}

We next verify that the target coefficient functions $a_{m,s}$ satisfy the
boundedness and Lipschitz conditions required by \citep[Theorem 2]{chen2022nonparametric}.
Since
\[
a_m(\mu)=G_m^{-1}b_m(\mu),
\]
we have
\[
\|a_m(\mu)\|_2
\le
\|G_m^{-1}\|_2\,\|b_m(\mu)\|_2.
\]
For each $s$,
\[
|(b_m(\mu))_s|
=
\left|
\frac1m\sum_{k=1}^m \mathbf{\Psi}(\mu,x_k)\psi_s(x_k)
\right|
\le U^2.
\]
Thus
\[
\|b_m(\mu)\|_2\le \sqrt r\,U^2.
\]
Using $\|G_m^{-1}\|_2\le 2/\gamma_r$, we obtain
\[
\|a_m(\mu)\|_2
\le
\frac{2\sqrt r\,U^2}{\gamma_r}.
\]
In particular,
\[
\|a_{m,s}\|_{L^\infty(\mathcal M)}
\le
\frac{2\sqrt r\,U^2}{\gamma_r},
\qquad s=1,\ldots,r.
\]

Similarly, for $\mu,\mu'\in\mathcal M$,
\[
|(b_m(\mu)-b_m(\mu'))_s|
\le
\frac1m\sum_{k=1}^m
|\mathbf{\Psi}(\mu,x_k)-\mathbf{\Psi}(\mu',x_k)|
\,|\psi_s(x_k)|
\le
UL_{\mathbf\Psi}d_{\mathcal M}(\mu,\mu').
\]
Therefore,
\[
\|b_m(\mu)-b_m(\mu')\|_2
\le
\sqrt r\,UL_{\mathbf\Psi}d_{\mathcal M}(\mu,\mu'),
\]
and hence
\[
\|a_m(\mu)-a_m(\mu')\|_2
\le
\frac{2\sqrt r\,UL_{\mathbf\Psi}}{\gamma_r}
d_{\mathcal M}(\mu,\mu').
\]
Thus each coordinate $a_{m,s}$ is Lipschitz with constant at most $\frac{2\sqrt r\,UL_{\mathbf\Psi}}{\gamma_r}$.

Applying \citep[Theorem 2]{chen2022nonparametric} to each coefficient function
$a_{m,s}$, we obtain 
\[
\mathbb{E}\left[
\int_{\mathcal{M}}
\big(a_{m,s}(\mu)-\widehat{a}_{m,s}(\mu)\big)^2\,dQ(\mu)
\right]
\le
c
\left(\frac{2\sqrt r\,U^2}{\gamma_r}\right)^2
\left(n^{-\frac{2}{2+d}}+\frac{D}{n}\right)\log^3n.
\]
Summing over $s=1,\ldots,r$ gives
\[
\mathbb{E}\!\left[
\|\mathbf{\Psi}_m^r-\widehat{\mathbf{\Psi}}_{n,m}^r\|_{L^2_Q(\mathcal{M};L^2(\Omega))}^2
\right]
\le
c\Gamma_r \frac{r^2\,U^4}{\gamma_r^2}
\left(n^{-\frac{2}{2+d}}+\frac{D}{n}\right)\log^3n.
\]

The right-hand side has no explicit $m$-dependence because the bound
$\|G_m^{-1}\|_2\le 2/\gamma_r$ is uniform once
Assumption~\ref{assump3}(iii) holds. The $m$-dependence is isolated in the
discretization error.
\end{proof}

\paragraph{$\bullet$ Combining the errors together.} We now combine the three estimates. The transport-map error is first converted to a $W_2$ error using Lemma~\ref{lemma:OT_distance}. The definition of the transport signature then converts the transport-map error into an $L^2_Q(\mathcal{M};L^2(\Omega))$ error for $\mathbf{\Psi}$. Finally, the latter is decomposed into the low-rank, discretization, and statistical terms estimated above.

\begin{proof}[Proof of Theorem~\ref{mainthm}]
The $2$-Wasserstein distance in~\eqref{eq:totalerror} can be bounded by the $\bar \rho$-weighted $L^2$ error between the corresponding transport maps as a result of Lemma~\ref{lemma:OT_distance}. Indeed, the coupling
$\big(\mathbf{T}(\mu,\cdot),\widehat{\mathbf T}_{n,m}^r(\mu,\cdot)\big)_\#\bar\rho$ 
gives 
\[
W_2^2\left(\rho(\mu),\widehat\rho_{n,m}^r(\mu)\right)
\le
\|\mathbf{T}(\mu,\cdot)-\widehat{\mathbf T}_{n,m}^r(\mu,\cdot)\|^2_{L^2_{\bar\rho}(\mathbb{R}^{D_{\bar\rho}};\mathbb{R}^{D_{\bar\rho}})}.
\]
Define
\[
\mathbf{\Phi}_m^r:=(I-\mathcal L)^{-1/2}\mathbf{\Psi}_m^r,
\qquad
\widehat{\mathbf{\Phi}}_{n,m}^r:=(I-\mathcal L)^{-1/2}\widehat{\mathbf{\Psi}}_{n,m}^r.
\]
Then
\begin{align} 
&\quad\mathbb{E}\!\left[\int_{\mathcal{M}}
W_2^2\!\left(\rho(\mu),\widehat{\rho}^r_{n,m}(\mu)\right)\,dQ(\mu) \right]
\nonumber \\
&\leq
\mathbb{E}\!\left[\int_{\mathcal{M}}
\|\nabla_x\mathbf{\Phi}(\mu,\cdot)
-\nabla_x\widehat{\mathbf{\Phi}}^r_{n,m}(\mu,\cdot)\|^2_{L^2_{\bar\rho}(\mathbb{R}^{D_{\bar\rho}};\mathbb{R}^{D_{\bar\rho}})}
\,dQ(\mu)\right]
\nonumber\\
&\leq
2\int_{\mathcal{M}}
\|\nabla_x\mathbf{\Phi}(\mu,\cdot)
-\nabla_x\mathbf{\Phi}_m^r(\mu,\cdot)\|^2_{L^2_{\bar\rho}(\mathbb{R}^{D_{\bar\rho}};\mathbb{R}^{D_{\bar\rho}})}
\,dQ(\mu)
\nonumber\\ 
&\quad +2\mathbb{E}\!\left[\int_{\mathcal{M}}
\|\nabla_x\mathbf{\Phi}_m^r(\mu,\cdot)
-\nabla_x\widehat{\mathbf{\Phi}}^r_{n,m}(\mu,\cdot)\|^2_{L^2_{\bar\rho}(\mathbb{R}^{D_{\bar\rho}};\mathbb{R}^{D_{\bar\rho}})}
\,dQ(\mu)\right]
\nonumber\\
&\leq
2\int_{\mathcal{M}}
\|\mathbf{\Psi}(\mu,\cdot)-\mathbf{\Psi}_m^r(\mu,\cdot)\|^2_{L^2(\Omega)}
\,dQ(\mu)
\nonumber +
2\mathbb{E}\!\left[\int_{\mathcal{M}}
\|\mathbf{\Psi}_m^r(\mu,\cdot)-\widehat{\mathbf{\Psi}}^r_{n,m}(\mu,\cdot)\|^2_{L^2(\Omega)}
\,dQ(\mu)\right]
\nonumber\\
&=
2\|\mathbf{\Psi}-\mathbf{\Psi}_m^r\|_{L^2_Q(\mathcal{M};L^2(\Omega))}^2
+
2\mathbb{E}\!\left[
\|\mathbf{\Psi}_m^r-\widehat{\mathbf{\Psi}}_{n,m}^r\|_{L^2_Q(\mathcal{M};L^2(\Omega))}^2
\right]
\nonumber\\
&\leq \label{eq:errordecomp}
4 \underbrace{\|\mathbf{\Psi}-\mathbf{\Psi}^{*r}\|_{L^2_Q(\mathcal{M};L^2(\Omega))}^2}_{\text{low-rank error}}
+
4\underbrace{\|\mathbf{\Psi}^{*r}-\mathbf{\Psi}_m^{r}\|_{L^2_Q(\mathcal{M};L^2(\Omega))}^2}_{\text{discretization error}} +
2\underbrace{\mathbb{E}\!\left[
\|\mathbf{\Psi}_m^r-\widehat{\mathbf{\Psi}}_{n,m}^r\|_{L^2_Q(\mathcal{M};L^2(\Omega))}^2
\right]}_{\text{statistical error}} \\
\nonumber
&\leq
\frac{4C^2(r+1)^2\sigma^2_{r+1}(\mathbf{\Psi})}{\theta^2}  
+ 4C_1\,\Gamma_r\,h_x(m)^2 
+2c\Gamma_r \frac{r^2\,U^4}{\gamma_r^2}
\left(n^{-\frac{2}{2+d}}+\frac Dn\right)\log^3n,
\end{align}
as desired.
The first inequality follows from the coupling induced by the two maps from the
reference measure $\bar\rho$; see Lemma~\ref{lemma:OT_distance}. The second and fourth inequalities follow from $\|A-C\|^2\le 2\|A-B\|^2+2\|B-C\|^2$. The third inequality uses the identity \eqref{eq:H1seminormbd}.
Moreover, the last part of the statement in Theorem~\ref{mainthm} follows directly from Lemma~\ref{lem:nu_bound_mult}. This completes the proof.
\end{proof}

\section{Numerical results}
\label{sec:num}

We demonstrate the proposed method on a two-dimensional continuity equation and one-dimensional Burgers equation, both with density-valued solutions. 
All numerical experiments are implemented in Python on a PC with an Apple M5 Pro chip and 24 GB of memory\footnote{The code to reproduce the numerical results in this paper is available at \url{https://github.com/jiajia-yu/rom_transport}.}. 

\subsection{Continuity equation}
\label{subsec:num cteqn}

Consider the continuity equation on $[0,1]^2$ with periodic boundary conditions
\begin{equation}
\left\{\begin{aligned}
    \partial_t\rho + \nabla\cdot(\rho v) = 0, &\quad(x,t)\in [0,1]^2\times(0,1],\\
    \rho(x,0) = \rho_0(x),&\quad x\in[0,1]^2.
\end{aligned}\right.    
\label{eq:FP}
\end{equation}
Here, $\rho_0$ is an isotropic Gaussian with mean $(0.5,0.5)$ and standard deviation $0.1$. 
Let the parameter be $\mu=(\alpha,\beta,t)$ and $v(x) = (\alpha\cos(2\pi x_2), \beta\sin(2\pi x_1))$.
Then $\rho(\mu)$ denotes the solution to \eqref{eq:FP} at time $t$ with the velocity field determined by $\alpha$ and $\beta$.

\subsubsection{Implementation}

For integers $n_x,n_t$, define $\Delta x=1/n_x$, $\Delta t=1/n_t$, spatial points $x_k=((k_1-\frac12)\Delta x,(k_2-\frac12)\Delta x)$ for $k=(k_1,k_2)$, and time points $t_l=l\Delta t$. Use $e_1=(1,0)$ and $e_2=(0,1)$ for index shifts.
For a space-time function $u$, its grid representation is also denoted by $u$, with $u_k^l\approx u(x_k,t_l)$.
For a probability density $\rho$, we normalize its grid representation so that $(\Delta x)^2\sum_{k_1,k_2=1}^{n_x}\rho_k=1$.

The full-order model (FOM) uses a semi-implicit Lax--Friedrichs scheme with periodic boundary conditions.
For a function $u$ defined on the grid, the elementary discrete differential operators are
\begin{equation*}
    (D_i u)_k = \frac{u_{k+e_i}-u_{k-e_i}}{2\Delta x}, \qquad
    (\Delta_i u)_k = \frac{u_{k+e_i}-2u_k+u_{k-e_i}}{(\Delta x)^2}, \quad i=1,2.
\end{equation*}
Periodic indexing is imposed by $u_{(0,k_2)}=u_{(n_x,k_2)}$, $u_{(k_1,0)}=u_{(k_1,n_x)}$, and $u_{(n_x+1,k_2)}=u_{(1,k_2)}$, $u_{(k_1,n_x+1)}=u_{(k_1,1)}$.
For the initial density $\rho^0$ and time-independent velocity $v=(v_1,v_2)$ defined on the grid, with numerical viscosity $\nu_{\rm num}=1$, the update reads
\begin{equation}
    \frac{\rho_k^{l+1}-\rho_k^l}{\Delta t} - \nu_{\text{num}}\Delta x(\Delta_1 \rho^{l+1}+\Delta_2 \rho^{l+1})_k + (D_1(\rho^{l+1}v_1^l)+D_2(\rho^{l+1}v_2^l))_k = 0.
\end{equation}

We sample 10 independent pairs of $(\alpha,\beta)$ from $U([-0.4,0.4]^2)$. 
For each pair, we solve the equation with the full-order model on a grid with $n_x=256,n_t=100$.
This yields 1000 samples of $(\mu,\rho(\mu))$ where $\mu=(\alpha,\beta,t)$ and each $\rho(\mu)$ is defined on $m=n_x^2$ grid points.
We then uniformly sample $n=400$ of these pairs without replacement as $\mu_j$ and $\rho(\mu_j)$, $j=1,\ldots,n$, to learn the reduced order model.

\paragraph{Offline stage (\romannum{1})}

To learn the reduced-order model, we choose the uniform distribution on $[0,1]^2$ as the reference $\bar{\rho}$ and compute the Kantorovich potential $\phi(\mu_j)$ that pushes $\bar{\rho}$ to $\rho(\mu_j)$ using the back-and-forth method~\citep{bfm}. 
Because $\bar{\rho}$ is uniform, the operator $\mathcal{L}$ reduces to the Laplacian, which allows us to compute $\psi(\mu_j)=(I-\mathcal{L})^{1/2}\phi(\mu_j)$ efficiently via the fast Fourier transform. The probability density $\rho$, the OT potential $\phi$, and the transport signature $\psi$ are all represented on $m=n_x^2$ grid points.

We remark that the back-and-forth method solves the unregularized optimal transport problem, in contrast to entropy-regularized approaches such as the Sinkhorn algorithm~\citep{sinkhorn}
Moreover, on regular domains like $[0,1]^2$, each iteration can be implemented using the fast Fourier transform and elementwise operations, avoiding the need to solve linear systems. Consequently, the method is both accurate and efficient.
In our experiment, computing OT potential for $n=400$ training data of size $256\times256$ takes 162.77 seconds. 

\paragraph{Offline stage (\romannum{2})}

By collecting the sampled grid-based transport signatures, we form the finite matrix $\Psi\in\mathbb{R}^{n\times m}$.
Applying the skeleton decomposition (Algorithm~\ref{2dmvalg_symm}) to $\Psi\Psi^\top$, we select row indices $I=\{i_1,\cdots,i_r\}$ with $r=10$ and form the basis matrix $\Psi_I\in\mathbb{R}^{r\times m}$.  The coefficients $a_m(\mu_j)$ and the mapping $\widehat{a}$ are learned as stated in Sec.~\ref{subsec: off2}.  The coefficients $a_m(\mu_j)$ are computed by least squares, and the mapping $\widehat{a}$ is represented by a neural network with two hidden layers of width 32 and ReLU activations. The neural network is trained by regression using the Adam optimizer with a learning rate of $5\times10^{-3}$ for $10^4$ epochs.

We use skeleton decomposition instead of singular value decomposition because it selects representative signatures from the data rather than forming global linear combinations of all snapshots. The resulting basis functions are therefore interpretable as transport signatures associated with actual parameter values. In addition, when the Gram matrix $G_m$ is well conditioned, the corresponding least-squares coefficient map is stable, which facilitates the subsequent regression step.

In skeleton decomposition, using a larger $r$ improves the low-rank approximation, but it increases the condition number of the Gram matrix $G_m$, degrades the stability of the coefficient map $a_m$. The choice of $r$ must balance the low-rank approximation error and the stability of the coefficient map $a$.

\paragraph{Online stage}

The online stage is implemented as described in Sec.~\ref{subsec: online}.
Analogous to the implementation of $(I-\mathcal{L})^{1/2}$ in offline stage (\romannum{1}), $(I-\mathcal{L})^{-1/2}$ is computed using the fast Fourier transform.
The pushforward computation is based on the identity
$$ \rho(T(A)) = \bar{\rho}(A),\text{ for any Borel set }A, \qquad T= I - \nabla\phi.$$
The source domain is adaptively partitioned into small cells based on the local compression or stretching induced by $T$.
At the center of each cell, the pushforward map $T$ is evaluated using a combination of bilinear interpolation and central differences. The mass within each cell is then assumed to be mapped to the image of its center under $T$.
Finally, the target density is constructed on a regular grid, with the transported mass redistributed to the four nearest grid cells using bilinear weighting.

\subsubsection{Verification of low-rank structure}

For each parameter $\mu_j$, $j=1,\ldots,n$, the full-order model and offline stage (\romannum{1}) produce $\rho(\mu_j)$, $\phi(\mu_j)$ and $\psi(\mu_j)$, each defined on $m=n_x^2$ grid points.
Collecting these values yields matrices of size $n\times m$ for $\rho$, $\phi$, and $\psi$, respectively.
We then compute the reduced SVD to obtain the leading singular values $\sigma_1,\ldots,\sigma_{50}$ for the three matrices.
We plot $\sigma_i/\sigma_1$ in Figure \ref{fig:lowrank2dfp}.
Figure \ref{fig:lowrank2dfp} shows that the leading rank-decay structure of the Kantorovich potential $\phi$ and transport signature $\psi$ is improved compared to $\rho$. The ratios $\sigma_{10}/\sigma_1$ for the $\phi$- and $\psi$-matrices are smaller than $10^{-2}$, while that of the $\rho$-matrix is greater than $10^{-2}$. 

\begin{figure}[h]
    \centering
    \includegraphics[width=0.5\linewidth]{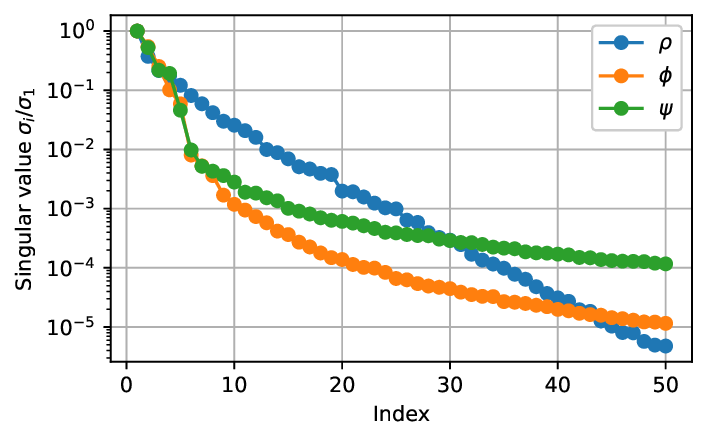}
    \caption{Singular value decay pattern of probability density $\rho$, OT potential $\phi$ and transport signature $\psi$. } %
    \label{fig:lowrank2dfp}
\end{figure}

\begin{remark}
Figure~\ref{fig:lowrank2dfp} shows that the singular values associated with the Kantorovich potential $\phi$ decay slightly faster than those of the transport signature $\psi$ starting around the eighth normalized singular value ($\sigma_8/\sigma_1$). Nevertheless, we construct the ROM using $\psi$ rather than $\phi$ for two reasons. First, the two quantities exhibit similar low-rank structure: their leading normalized singular values are nearly identical, and both decay substantially faster than those associated with $\rho$. Thus, at the ranks relevant to our reduced approximation, working with $\phi$ offers no significant advantage in terms of low-rankness. Second, and more importantly, approximation of the density in the $W_2$ metric requires control of the corresponding transport map and hence of the $H^1$ seminorm of the error in the Kantorovich potential. A reduced approximation constructed directly for $\phi$ based on its function values naturally provides control in an $L^2$-type metric, but does not directly control this $H^1$ seminorm. In contrast, our choice of the transport signature $\psi$ is specifically designed so that its approximation error provides the required control of the transport map and, consequently, of the induced $W_2$ error. Therefore, the slightly faster singular-value decay observed for $\phi$ at higher indices does not outweigh the theoretical advantage of constructing the ROM in terms of $\psi$.
\end{remark}

\subsubsection{Evaluation of reduced-order model}

We evaluate the accuracy and generalization of the ROM using the training data, testing data, and an unseen parameter pair. For each snapshot, we measure the discrepancy between the FOM solution and the ROM approximation using the Wasserstein-2 distance.
For a dataset containing N snapshots, we report the average error.
We also compare FOM and ROM solutions visually. In Figures~\ref{fig:training2dfp} and~\ref{fig:testing2dfp}, the randomly selected FOM solutions are shown in the top rows and the corresponding ROM approximations in the bottom rows, with the $W_2$ error indicated at the top of each column.

We first evaluate the ROM on the $n=400$ snapshots used for training. The average $W_2$ error is 0.0143, and Figure~\ref{fig:training2dfp} shows 10 randomly selected examples. We then evaluate the ROM on 600 testing snapshots sampled from the same parameter range but not used during training. The average $W_2$ error is 0.0145, with 10 randomly selected examples shown in Figure~\ref{fig:testing2dfp}. The testing error is nearly identical to the training error, indicating that the ROM generalizes well to held-out samples within the training parameter regime.

We further examine the ROM for a newly sampled parameter pair $(\alpha,\beta)=(0.0134,0.0898)$. The corresponding FOM trajectory is computed on the same spatial grid, $n_x=256$, with a refined temporal grid, $n_t=200$. The average $W_2$ error over the trajectory is 0.0195. Figure~\ref{fig:newpar2dfp1} shows the location of the new parameter relative to the training samples, the evolution of the $W_2$ error in time, and the FOM and ROM solutions at t=0.1,0.2,\ldots,1. Although the error is moderately larger than those on the training and testing datasets, it remains of the same order of magnitude, demonstrating that the ROM maintains good accuracy for a parameter not encountered during training.

Overall, the ROM achieves comparable average $W_2$ errors on the training, testing, and unseen-parameter cases, with average errors of 0.0143, 0.0145, and 0.0195, respectively. Quantitatively, these results indicate that the ROM captures the underlying solution structure without substantial overfitting to the training snapshots. Visually, the ROM solutions closely reproduce both the location and shape of the density profiles observed in the FOM solutions. Together, these results demonstrate that the ROM accurately captures the main transport features of the solution and performs well for interpolation within the sampled parameter regime.

\begin{figure}[h]
    \centering
    \includegraphics[width=\linewidth]{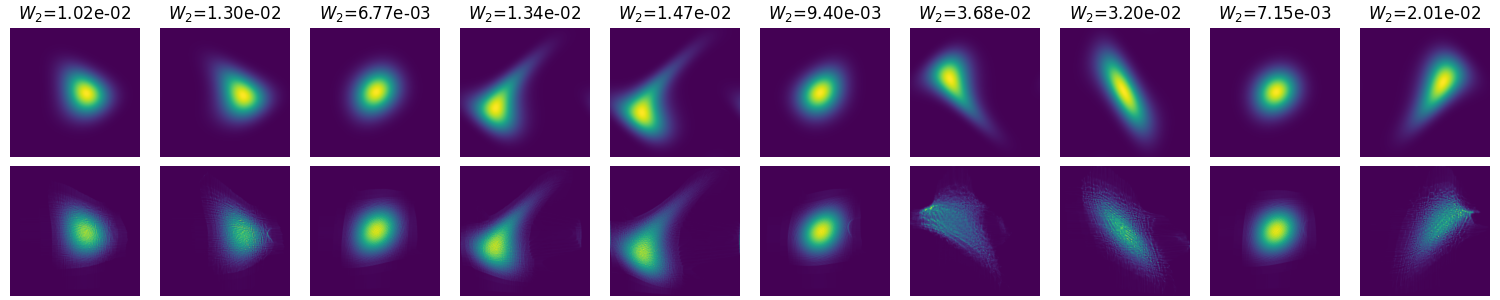}
    \caption{Continuity equation: Samples of training data. Top row: FOM results; bottom row: ROM results.}
    \label{fig:training2dfp}
\end{figure}

\begin{figure}[h]
    \centering
    \includegraphics[width=\linewidth]{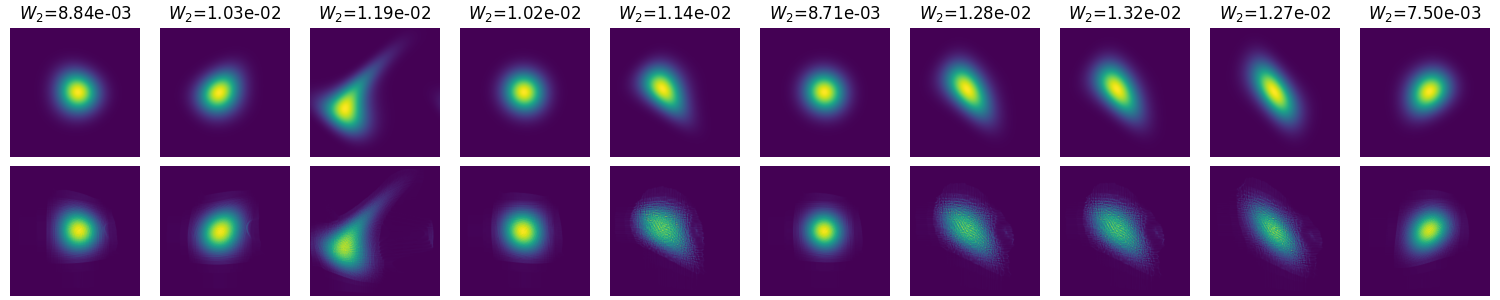}
    \caption{Continuity equation:  Samples of testing data. Top row: FOM results; bottom row: ROM results.}
    \label{fig:testing2dfp}
\end{figure}

\begin{figure}[h]
    \centering
    \includegraphics[height=0.35\linewidth]{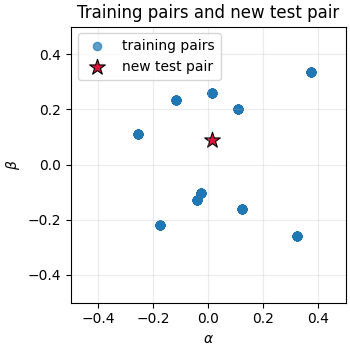}\hspace{1em}
    \includegraphics[height=0.35\linewidth]{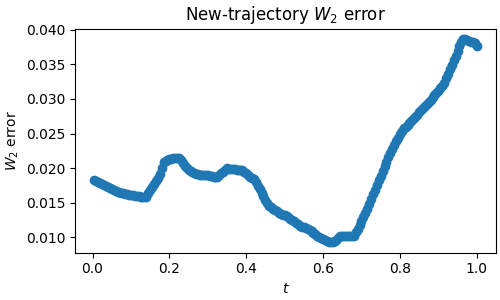}
    \includegraphics[width=\linewidth]{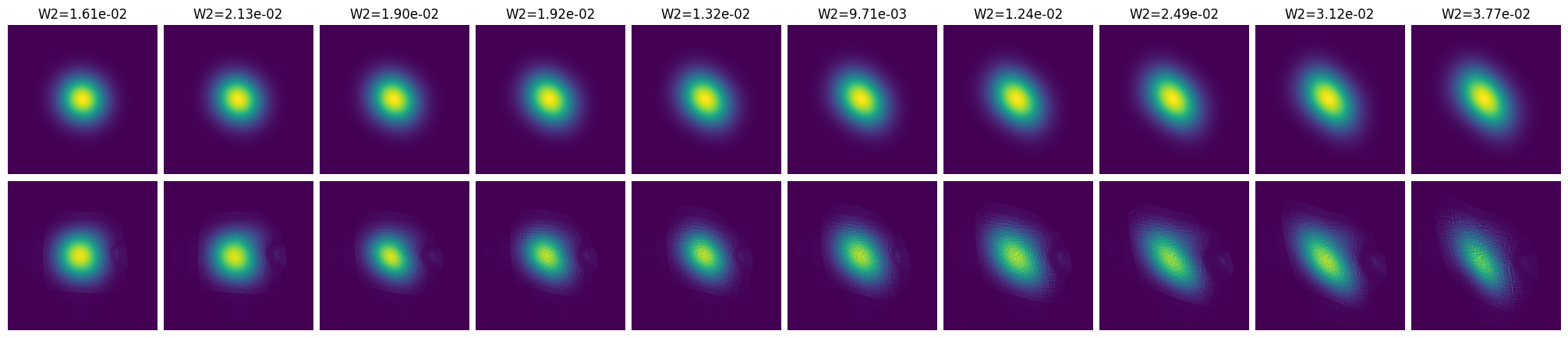}
    \caption{Numerical results for new parameters for the continuity equation. Top left: training parameters and the newly sampled parameter. Top right: $W_2$ error versus $t$ for the new sample. Bottom: $\rho(\mu)$ of an unseen parameter $(\alpha,\beta)$ at $t=0.1,t=0.2,\cdots,1.0$ (left to right) by FOM (top row) and ROM (bottom row).
    }
    \label{fig:newpar2dfp1}
\end{figure}

\subsection{Burgers equation}

In this example, we aim to test our method on equations on nonlinear equations and possibly with diffusion. 
We consider the one-dimensional Burgers equation for a probability density
\begin{equation}\label{eq:burgers-density}
\begin{cases}
    \partial_t \rho + \partial_x f(\rho)
    = \nu \partial_{xx}\rho,
    & x\in \mathbb{T},\quad t\in[0,T],\\
    \rho(x,0;x_c,w)
    = \rho_0(x;x_c,w),
    & x\in \mathbb{T},
\end{cases}
\end{equation}
where $f(\rho)=\frac{\rho^2}{2}$, $\mathbb{T}=[0,1]$ denotes the periodic unit interval and $T=0.02$. 
The initial condition is a smooth top-hat density. Define
\begin{equation}\label{eq:smooth-tophat}
    g(x;x_c,w)
    :=
    \frac12
    \left[
    \tanh\left(\frac{x-(x_c-w/2)}{\delta}\right)
    -
    \tanh\left(\frac{x-(x_c+w/2)}{\delta}\right)
    \right],
\end{equation}
with transition width $\delta=10^{-3}$.
We then set
\begin{equation}\label{eq:burgers-initial-density}
    \rho_0(x;x_c,w)
    :=
    \eta
    +
    (1-\eta)
    \frac{g(x;x_c,w)}
    {\displaystyle \int_0^1 g(y;x_c,w)\,dy},
    \qquad
    \eta=10^{-2}.
\end{equation}
Since $|\mathbb{T}|=1$, this normalization gives $\int_0^1 \rho_0(x;x_c,w)\,dx=1$ and the positive constant $\eta$ ensures that $\rho_0$ is bounded away from zero. The solution of \eqref{eq:burgers-density} is therefore mass-preserving and remains a probability density.

For the reduced-order modeling task, we define
\[
    \rho(\mu)(x)
    :=
    \rho(x,T;x_c,w),
    \qquad
    \mu=(x_c,w),
\]
where $x_c$ is the initial center of the density, and $w$ is its initial width. 
We test our method for both inviscid ($\nu=0$) and viscous ($\nu=0.1$) cases.
For the inviscid case ($\nu=0$), the density snapshots $\rho(\mu)$ are expected to have relatively slow linear $n$-width decay, whereas the corresponding transport signatures turn out to admit a more compact low-rank representation.

\subsubsection{Implementation}

The full-order model is constructed using the method-of-lines approach, with a formally second-order conservative discretization in space coupled with a slope limiter to control spurious oscillations. For the temporal discretization, we employ the classical third-order strong stability-preserving Runge--Kutta (SSPRK3) method \citep{GottliebShuTadmor2001}, with the time step $\Delta t=T/n_t$ chosen to ensure stability. 
In what follows, we focus exclusively on the spatial discretization.

Let $n_x$ be a positive integer and $\Delta x=1/n_x$. We start with a uniform grid, with $x_k=\left(k-\frac{1}{2}\right)\Delta x$ as the centers of the grid elements, 
and let $\rho_k(t)\approx\rho(x_k,t)$. The semi-discrete in space scheme is given as
\begin{equation}
    \frac{d \rho_k}{d t}
    =
    -\frac{\hat{f}_{k+1/2}-\hat{f}_{k-1/2}}{\Delta x}
    +\nu\frac{\rho_{k+1}-2\rho_k+\rho_{k-1}}{(\Delta x)^2}.
\end{equation}
Here, the convective term is discretized in conservative form using the Rusanov numerical flux
\begin{equation}
    \hat{f}_{k+1/2}
    = \frac{1}{2}
    \left( f(\rho^L_{k+1/2})+f(\rho^R_{k+1/2}) \right)
    -\frac{1}{2}a_{k+1/2}
    \left( \rho^R_{k+1/2}-\rho^L_{k+1/2} \right),
\end{equation}
where $a_{k+1/2}
    = \max( |f'(\rho^L_{k+1/2})|, |f'(\rho^R_{k+1/2})| )
    = \max (|\rho^L_{k+1/2}|,|\rho^R_{k+1/2}|).$

The left and right states at the grid interfaces are obtained using the
Monotone Upstream-centered Scheme for Conservation Laws (MUSCL) reconstruction~\citep{vanLeer1977}, 
\begin{equation}
    \rho^L_{k+1/2}
    = \rho_k+\frac{1}{2}\sigma_k,\quad
    \rho^R_{k+1/2}
    = \rho_{k+1}-\frac{1}{2}\sigma_{k+1},
\end{equation}
where the slopes are determined by the monotonized central (MC) limiter~\citep{vanLeer1977,LeVeque2002},
\begin{equation}
    \sigma_k
    =\operatorname{minmod}
    \left(
        2\Delta_k^-,
        \frac{\Delta_k^-+\Delta_k^+}{2},
        2\Delta_k^+
    \right).
    \end{equation}

Here,  $\Delta_k^-=\rho_k-\rho_{k-1},
    \Delta_k^+=\rho_{k+1}-\rho_k$,
and
\begin{equation}
    \operatorname{minmod}(a,b,c)
    =
    \begin{cases}
        \operatorname{sgn}(a)\min(|a|,|b|,|c|),
        & a,b,c \text{ have the same sign},\\
        0,
        & \text{otherwise}.
    \end{cases}
\end{equation}

We sample 100 parameter pairs $(x_c,w)$ uniformly from $[0.18,0.28]\times[0.08,0.16]$. For each pair, we solve the equation using the full-order model on a grid with $n_x=1024$, with the number of time steps $n_t$ chosen to satisfy the stability restriction. 
This yields $n=100$ training samples ${(\mu_j,\rho(\mu_j))}_{j=1}^n$ for constructing the reduced-order model, where $\mu=(x_c,w)$ and each $\rho(\mu)$ is represented on $m=n_x$ grid points.

The offline stages (\romannum{1})–(\romannum{2}) and the online stage follow the procedure described in Section~\ref{subsec:num cteqn}, with the following specifications. We choose the uniform distribution on $[0,1]$ as the reference density $\bar{\rho}$ and compute the Kantorovich potential $\phi(\mu_j)$ using the exact formula based on the cumulative distribution function in~\eqref{eq:kan 1d exact}. For the skeleton decomposition, we set $r=3$. To learn the parameter-to-coefficient map, we use the same network architecture and train it with the Adam optimizer for $10^4$ epochs with a learning rate of $1\times10^{-3}$. In the online stage, the pushforward is computed by first determining the image of each grid interval under the transport map and then distributing the mass of that source interval uniformly over its image.

\subsubsection{Verification of low-rank structure}

Similar to Section~\ref{subsec:num cteqn}, we compute the reduced SVD to obtain the leading singular values $\sigma_1,\ldots,\sigma_{50}$ for density $\rho$, Kantorovich potential $\phi$ and transport signature $\psi$.
We plot $\sigma_i/\sigma_1$ in Figure \ref{fig:lowrank1dbgs}.
Figure \ref{fig:lowrank1dbgs} shows that for both inviscid and viscous cases, the leading rank-decay structure of the Kantorovich potential $\phi$ and transport signature $\psi$ is improved compared to $\rho$, with the improvement being more pronounced in the inviscid case. 

\begin{figure}[h]
    \centering
    \includegraphics[width=0.48\linewidth]{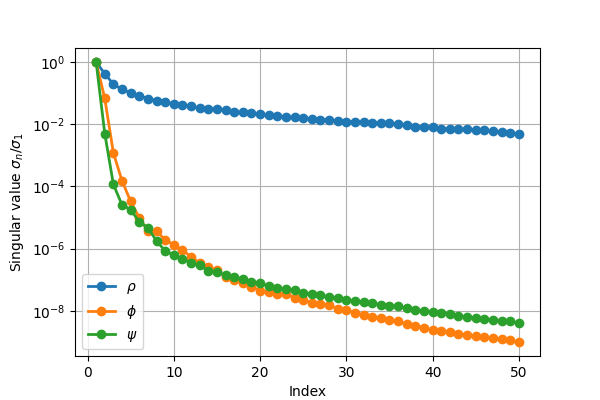}
    \includegraphics[width=0.48\linewidth]{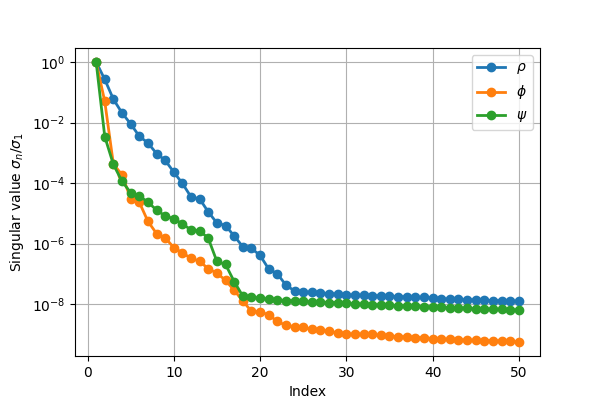}
    \caption{Burgers equation: singular value decay pattern of probability density $\rho$, OT potential $\phi$ and transport signature $\psi$ (left: $\nu=0$, right: $\nu=0.1$).}
    \label{fig:lowrank1dbgs}
\end{figure}

\subsubsection{Evaluation of reduced-order model}

In addition to the $n=100$ training data, we test our ROM on a separate testing dataset comprising 100 newly sampled parameter pairs drawn from the same range but not used during training.

For each dataset, we compute the Wasserstein-2 distance between the ROM and FOM results, $W_2(\rho(\mu),\widehat{\rho}(\mu))$, and report the average $W_2$ distance. We also plot the FOM and ROM results for 5 sampled snapshots for visual comparison. The training results are shown in Figure~\ref{fig:training1dbgs} and the testing results in Figure~\ref{fig:testing1dbgs}.

For both the inviscid and viscous cases, the ROM accurately captures the sharp gradient locations and preserves the positivity of the density.
In the inviscid case, the singular value decay pattern and the small $W_2$ error clearly demonstrate the advantage of our method in capturing the low-rank structure.
In the viscous case, although the improvement in the low-rank structure is less pronounced, it is still present. Moreover, our method naturally preserves the positivity of the density and the total mass, which is an advantage over directly applying standard ROMs to the density. Note that no mechanism is built into our ROM to control numerical oscillations as in the FOM; overshoots and undershoots can be observed in our ROM solutions in the inviscid case. Note that no mechanism is built into our ROM to control numerical oscillations as in the FOM, and overshoots and undershoots can be observed in our ROM solutions in the inviscid case.

\begin{figure}[h]
    \centering
    \includegraphics[width=\linewidth]{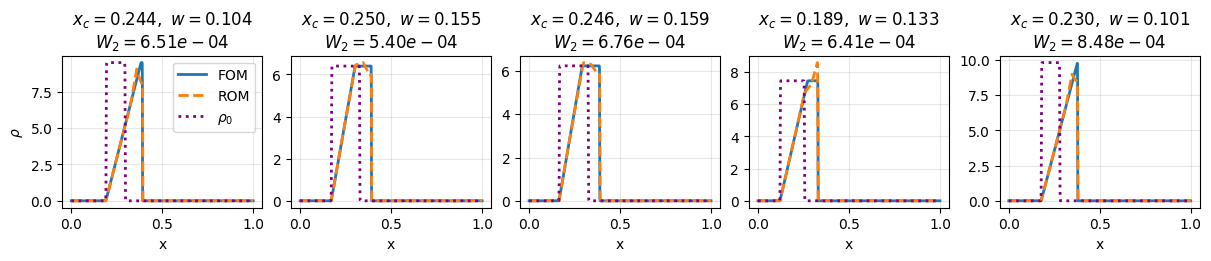}
    \includegraphics[width=\linewidth]{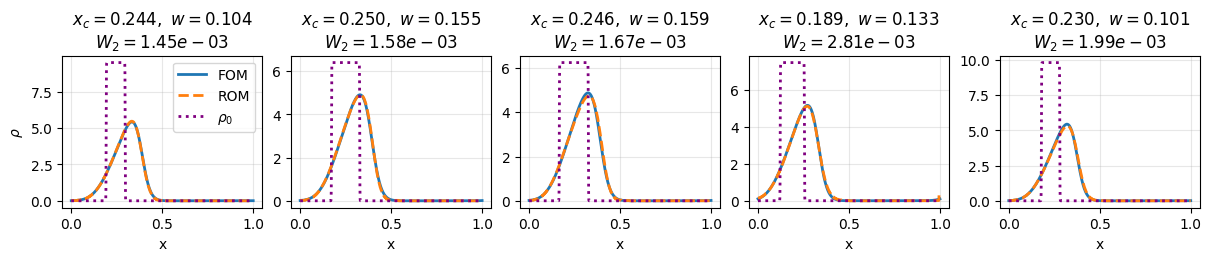}
    \caption{Burgers equation: samples of training data (top: $\nu=0$, average $W_2$ error $7.35\times 10^{-4}$; bottom: $\nu=0.1$, average $W_2$ error $2.11\times 10^{-3}$).}
    \label{fig:training1dbgs}
\end{figure}

\begin{figure}[h]
    \centering
    \includegraphics[width=\linewidth]{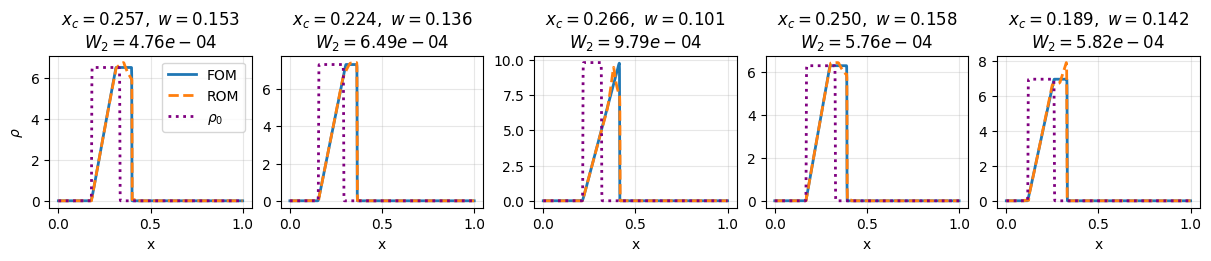}
    \includegraphics[width=\linewidth]{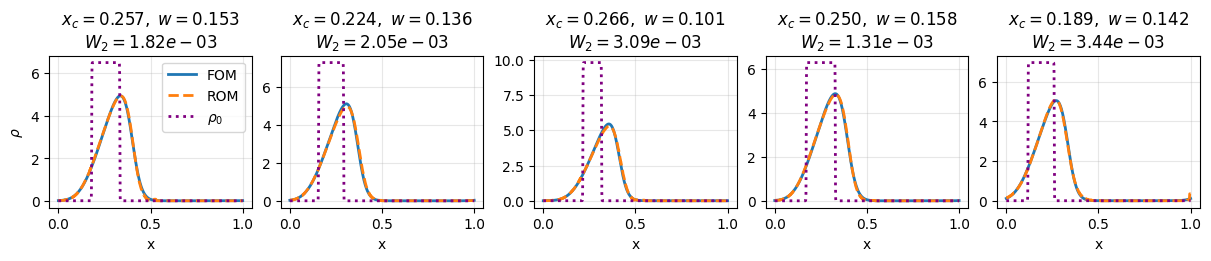}
    \caption{Burgers equation: samples of testing data (top: $\nu=0$, average $W_2$ error $8.46\times 10^{-4}$ bottom: $\nu=0.1$, average $W_2$ error $2.52\times 10^{-3}$).}
    \label{fig:testing1dbgs}
\end{figure}

\section{Conclusion and Discussion}\label{sec:conclusions}

In this work, we introduced an optimal-transport-based reduced-order modeling framework for parametrized density-valued PDEs. For transport-dominated probability distributions, it is more effective to approximate how mass moves from a reference state than to approximate the densities themselves. The skeleton decomposition provides a compact set of representative transport modes, while the learned coefficient map gives a practical non-intrusive surrogate for new parameter values. The accompanying error analysis justifies this construction by linking the approximation error in signature space to the Wasserstein distance between the reconstructed distributions. The numerical experiments indicate that this representation can be effective for the solution of continuity equations and Burgers equations. 
\paragraph{Limitation and future directions.}

One limitation of our framework is its reliance on a single fixed reference distribution, which may be less effective for complex or strongly multimodal solution families. Future directions include developing adaptive or local reference strategies, controlling numerical oscillations near discontinuities or sharp gradients, extending the framework to higher-dimensional problems where optimal transport computations are more demanding, and generalizing the methodology to equations with non-conservative mass dynamics.

\section*{Acknowledgment}
All the authors acknowledge the support of the ICERM workshop ``Empowering a Diverse Computational Mathematics Research Community" on July 22 - August 2, 2024, where this work was initiated.

JH is partially supported by National Science Foundation under Grant No. DMS-2409858. FL is partially supported by Air Force Office of Scientific Research under Grant No.~FA9550-26-1-0003. ST is partially supported by National Science Foundation under Grant No. DMS-2529292. YY is partially supported by National Science Foundation under Grant No.~DMS-2409855 and No.~DMS-2540324 and Office of Naval Research under Award No.~N00014-24-1-2088. ZS is partially supported by SIAM Postdoctoral Support Program.

\bibliography{ROM}
\bibliographystyle{plain}

\end{document}